%% file: main_arxiv.tex
\documentclass[hidelinks,onefignum,onetabnum]{siamart251216}

\input{shared_arxiv}

\ifpdf
\hypersetup{
  pdftitle={Hierarchical low-rank compression exploiting symmetric boundary integral equations in transmission scattering},
  pdfauthor={Y. Matsumoto, T. Maruyama, Q. Ma and R. Yokota}
}
\fi

\usepackage{physics}
\usepackage{booktabs}
\usepackage{siunitx}
\usepackage{multirow}
\usepackage{mathtools}
\DeclarePairedDelimiter{\bilinear}{\langle}{\rangle}
\mathtoolsset{showonlyrefs}
\usepackage{bm}
\usepackage{capt-of}

\newtheorem{example}{Example}%

\begin{document}

\maketitle

% REQUIRED
\begin{abstract}
  We propose an approach using matrix symmetry to accelerate hierarchical low-rank compression for wave transmission scattering.
  For Helmholtz transmission problems, we introduce a symmetric Müller-like boundary integral equation formulation.
  Although not of the second kind, this formulation creates no difficulties for direct solvers.
  A one-point corrected quadrature in a Nyström discretization preserves the complex symmetry of the discretized system.
  This formulation is integrated into a variant of a hierarchical semiseparable direct solver.
  Within this framework, the computational workload of the low-rank approximation in the proxy method is halved by exploiting the symmetric relationship between the left and right coefficients.
  Furthermore, we extend this symmetry-exploiting approach to elastic transmission problems by using a symmetric Galerkin boundary element method.
  Numerical experiments on a multi-node computing cluster demonstrate that the symmetric Müller-like formulation for Helmholtz scattering achieves on average a 1.41-fold speedup compared to the non-symmetric baseline while maintaining third-order convergence.
  Furthermore, the symmetry-aware solver for elastic scattering is more than 1.5 times faster.
\end{abstract}

% REQUIRED

% REQUIRED

\section{Introduction}
Boundary integral equations are powerful tools for analyzing wave transmission scattering in physics and engineering, with applications including acoustics \cite{ganesh2024fast}, electromagnetics \cite{nick2024numerical}, and elastodynamics \cite{dominguez2022boundary}.
Since the coefficient matrix of the discretized linear equations stemming from boundary integral equations is dense, the combinations of the fast multipole method (and its variants) and Krylov subspace methods have been actively researched for solving large-scale problems \cite{Ying2009, 2019Abduljabbar, liu2024massive}.
In this context, numerous boundary integral equations have been investigated to not only guarantee unique solvability but also accelerate the convergence of iterative solvers.
As an early example, the (standard) M\"uller method formulates transmission problems as a Fredholm boundary integral equation of the second kind \cite{muller1969foundations}.
The integral operator associated with a second-kind equation can be expressed as the sum of a constant multiple of the identity operator and a compact operator; therefore, it possesses spectral properties desirable for iterative solvers \cite{kress2014linear}.
To obtain linear equations of the second kind for the Helmholtz transmission problem,
 the single boundary integral equation approach utilizing a product of layer potentials via indirect potentials was proposed \cite{kleinman1988single}, and
alternative regularized formulations based on operator products have also been proposed \cite{boubendir2015integral}.
Several complementary techniques for enhancing the spectral properties of first-kind formulations through preconditioning have also been developed \cite{Antoine01102008, NIINO201266, VANTWOUT2022111229}.

However, there is a growing trend toward fast direct solvers \cite{martinsson2019book}.
Since such solvers are applicable even to relatively ill-conditioned problems \cite{Greengard2009fast},
 employing boundary integral equations of the second kind is not a strict requirement.
 This difference between direct and iterative solvers motivates the formulation of alternative boundary integral equations tailored specifically for direct solvers.
 Within numerical solutions of integral equations, evaluating matrix entries frequently dominates the computational time compared to algebraic operations such as matrix multiplication or factorization.
 Consequently, utilizing efficient low-rank approximation \cite{bebendorf2000approximation} techniques offers a promising solution.
Clearly, if a coefficient matrix possesses symmetry, the low-rank approximation workload can be cut in half.
 Note that for frequency domain wave scattering, the resulting coefficient matrix is complex-valued; thus, we consider complex symmetry instead of the Hermitian property.
 Although various symmetric formulations have been developed prior to the present work \cite{hsiao2011system, laliena2009symmetric}, including but not limited to the symmetric Galerkin boundary element method \cite{bonnet1998symmetric, MARUYAMA202511-251219}, several challenges have hindered their integration for direct solvers.
 These include the appearance of fictitious frequencies \cite{CHEN1998529} on the real axis,
 increased degrees of freedom, and the explicit use of hypersingular kernels.

We propose a symmetric M\"uller-like boundary integral equation.
 Unlike the standard M\"uller method, this formulation is not of the Fredholm second kind; consequently, it converges slowly when solved using Krylov subspace methods, making it less attractive for iterative solvers, but it is an appealing candidate for direct solvers.
This proposed M\"uller-like formulation involves only weakly singular kernels, maintains the same number of unknowns as the conventional formulation, and is derived as a modification of the coupling parameter based on Green's representation theorem; consequently, fictitious eigenvalues do not appear on the real axis.
Furthermore, since the physical unknowns on the boundary are obtained directly as solutions of the boundary integral equation, this approach offers simplicity in engineering applications.
We note that the symmetric M\"uller-like boundary integral equation can be obtained by a constant scaling and swapping the columns of the operator in the formulation introduced in \cite{misawa2012jascome}.

Another contribution of this paper is the discretization and integration into a solver. Specifically, we discretize the symmetric M\"uller-like equation while preserving its complex symmetry.
 We then integrate it into a direct solver based on a variant of the hierarchical semiseparable (HSS) representation \cite{chandrasekaran2005Calcolo}.
Although the HSS representation achieves exceptionally high parallel efficiency \cite{rouet2016distributed},
 the low-rank approximation of all off-diagonal blocks limits its applicability to higher-order discretizations \cite{matsumoto2026proxy}.
Since the symmetric M\"uller-like equation involves at most weakly singular kernels, high-order convergence can be expected even with a one-point corrected quadrature in a Nystr\"om discretization; therefore, there is no obstacle to its integration into the HSS representation.
We discuss a method for reducing computational workload by exploiting symmetry within the proxy method, a typical low-rank approximation approach for boundary integral equations \cite{MARTINSSON20051, martinsson2007fast}.
This acceleration method using symmetry for low-rank approximation serves as a general framework that can be extended directly from Helmholtz scattering to elastodynamics.

Although recent interest has shifted toward $\mathcal{H}^2$-matrix-based direct solvers \cite{minden2017recursive, ma2024inherently, boukaram2026linear}, research into HSS-based direct solvers remains highly valuable due to persistent challenges in parallel efficiency, accuracy, and implementation complexity.
 Furthermore, the potential extension of the HSS representation to advanced solvers, including the hierarchical interpolative factorization, among others \cite{corona2015n, k_l_ho2016hierarchical, kandappan2023hodlr2d}, that incorporate an additional re-compression phase or a geometrically adjusted admissibility condition offers an appealing avenue for future development.

The remainder of this manuscript is organized as follows.
 Section \ref{sec:helmholtz_scattering} defines the Helmholtz transmission problem.
 Section \ref{sec:symm} discusses the required discretization based on the symmetric properties of the fundamental solution and subsequently introduces the proposed M\"uller-like boundary integral equation.
 Section \ref{sec:fds} describes the low-rank approximation using the proxy method under symmetry, as well as the corresponding HSS-variant fast direct solver.
 In Section \ref{sec:extension}, the method developed for the Helmholtz transmission problem is extended to the elastic transmission problem.
 Section \ref{sec:numerical_demo} presents comprehensive numerical results to validate the accuracy and efficiency of the proposed method.
 Finally, Section \ref{sec:conclusions} concludes the paper.

\section{Helmholtz transmission problems} \label{sec:helmholtz_scattering}
We briefly describe the Helmholtz transmission problem.
The setting is illustrated in Fig. \ref{fig:domain}.
Let $\Omega_1$ be an open and bounded subset of $\mathbb R^d$ for $d = 2, 3$.
We assume that the exterior $\Omega_0 = \mathbb R^d \setminus {\overline \Omega_1}$ is connected,
whereas $\Omega_1$ is allowed to be disconnected.
Let $\Gamma=\partial \Omega_1$ be the interface between the two regions.
We assume that $\Gamma$ is a Jordan closed curve of at least class $C^2$.
Let $\varepsilon_i > 0$, $\omega > 0$, and $k_i = \omega \sqrt{\varepsilon_i} > 0$ be constants
representing the material constant, angular frequency, and wavenumber, respectively, in $\Omega_{i}$ ($i=0,1$).
Further, let $u^{\mathrm{in}}$ be the incident wave field
that solves the Helmholtz equation with wavenumber $k_0$ in $\mathbb{R}^d$.
We consider the following transmission problem:
find the scattered fields $u_0^{\mathrm{sc}}$ in $\Omega_0$ and $u_1^{\mathrm{sc}}$ in $\Omega_1$ such that
\begin{align}
  %(P) \quad
  \begin{dcases}
    \Delta u_0^{\mathrm{sc}} + k_0^2 u_0^{\mathrm{sc}} = 0, & x \in \Omega_0,
    \\
    \Delta u_1^{\mathrm{sc}} + k_1^2 u_1^{\mathrm{sc}} = 0, & x \in \Omega_1,
    \\
    \lim_{h \to 0_+}{\qty(u_0^{\mathrm{sc}} + u^{\mathrm{in}})(x + h\nu(x))} = \lim_{h \to 0_+}{ u_1^{\mathrm{sc}}(x - h \nu(x)) } \,\, \qty(= u(x)),
    & x \in \Gamma,
    \\
    \frac{1}{\varepsilon_0} \nu(x) \cdot \qty{ \lim_{h \to 0_+} \nabla \qty(u_0^{\mathrm{sc}} + u_0^{\mathrm{in}}) (x + h\nu(x))) } \\
    \qquad \qquad = \frac{1}{\varepsilon_1} \nu(x) \cdot \qty{ \lim_{h \to 0_+} \nabla u_1^{\mathrm{sc}}(x - h \nu(x)) } \,\, \qty(= q(x)),
    & x \in \Gamma,
    \\
      \lim_{|x| \to \infty}{|x|}^{\frac{d-1}{2}}\left( \frac{\partial u_0^{\mathrm{sc}}(x)}{\partial |x|} - ik_0 u_0^{\mathrm{sc}}(x) \right) = 0,
  \end{dcases}
  \label{eq:problem}
\end{align}
where $i$ is the imaginary unit and the unit normal vector $\nu$ on $\Gamma$ is directed into $\Omega_0$.
We define the two functions $u, q$ on the boundary $\Gamma$ using
the boundary conditions in \eqref{eq:problem}.
\begin{figure}[tb]
  \centering
  \includegraphics[width=0.3\linewidth]{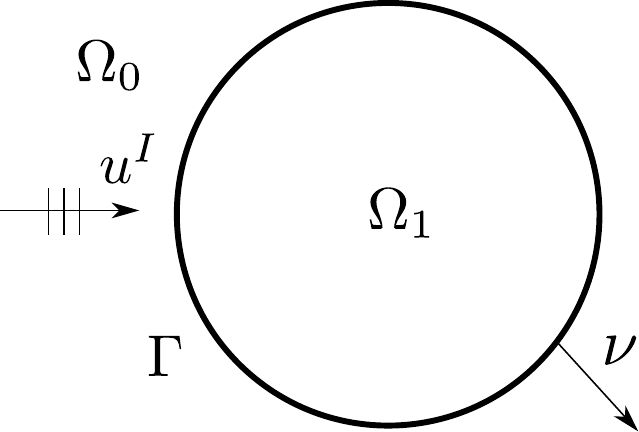}
  \caption{Diagram of Helmholtz transmission problem}
  \label{fig:domain}
\end{figure}
It is known that 
transmission problem \eqref{eq:problem} is uniquely solvable \cite{kress1977transmission}.

\section{Symmetric formulations of boundary integral equations} \label{sec:symm}
The procedure for symmetrizing the coefficient matrix of a discretized linear equation in the Nystr\"om method is based on the symmetry of the fundamental solution of the Helmholtz equation.
Although this symmetry is also used in the context of the symmetric Galerkin boundary element method \cite{sutradhar2008symmetric},
we will show that the same symmetric property works well in the Nystr\"om method if we use a quadrature formula that satisfies an appropriate condition. 
\subsection{Boundary integral operators and their Nystr\"om discretization} \label{sec:nystrom}
We first define the following boundary integral operators:
\begin{align}
  (S_k\varphi)(x) &= \int_{\Gamma} G_k(x,y)\varphi(y) \dd s(y), \label{eq:slp}
  \\
  (D_k\varphi)(x) &= \int_{\Gamma} \frac{\partial G_k}{\partial \nu(y)}(x,y)\varphi(y) \dd s(y),
  \\
  (D_k^* \varphi)(x) &= \int_{\Gamma} \frac{\partial G_k}{\partial \nu(x)}(x,y)\varphi(y) \dd s(y),
  \\
  (N_k\varphi)(x) &= \text{FP} \int_{\Gamma} \frac{\partial^2 G_k}{\partial \nu(x) \partial \nu(y)}(x,y)\varphi(y) \dd s(y), \label{eq:hyper_singular}
\end{align}
for all $x\in\Gamma$ and 
function $\varphi : \Gamma \to \mathbb{C}$, where
$\frac{\partial f}{\partial\nu(z)}$ is the normal derivative with respect to $\nu(z)$,
defined as
\begin{equation}
\frac{\partial f}{\partial\nu(z)} = \lim_{h\to 0_+} \nu(z)\cdot \nabla f(z + h \nu(z)), \quad z \in \Gamma.
\end{equation}
Here, $G_k(x, y)$ is the fundamental solution of the Helmholtz equation, given by
\begin{align}
  G_k(x, y) =
  \begin{dcases}
    \frac{i}{4} H_{0}^{(1)} (k |x-y|) & (d = 2), \\
    \frac{e^{i k |x - y|}}{4 \pi |x-y|} & (d = 3),
  \end{dcases}
\end{align}
for $x, y \in \mathbb{R}^d$ with $x \neq y$, where $H_{0}^{(1)}(z)$ is the zeroth-order Hankel function of the first kind.
In \eqref{eq:hyper_singular}, ``FP'' means the finite part of a divergent integral.

It is well known that the fundamental solution $G_k(x, y)$ has the following symmetry \cite{sutradhar2008symmetric}:
  \begin{align}
    G_k(x, y) = G_k(y, x), \quad
    \pdv{G_k(x, y)}{\nu(y)} = \pdv{G_k(y, x)}{\nu(y)}, \quad
    \pdv{G_k(x, y)}{\nu(x)}{\nu(y)} = \pdv{G_k(y, x)}{\nu(y)}{\nu(x)},
  \end{align}
for all $x, y$ $(x \neq y)$ on the boundary $\Gamma$.
However, even if the kernel in the integral operator is symmetric,
that does not guarantee that the matrix obtained by discretizing the operator is symmetric.
As a two-dimensional example, consider the single-layer potential \eqref{eq:slp} defined via a parametrization of $\Gamma$:
\begin{align}
  S_k \varphi(x) = \int_\Gamma G_k(x, y) \, \varphi(y) \, \dd s(y) 
  = \int_{-\pi}^{\pi} G_k(x, y(\theta)) \, \varphi(y(\theta)) \, |y^{\prime}(\theta)| \, \dd \theta,
\end{align}
where $|y^{\prime}(\theta)|$ is the Jacobian of the parametrization,
 and this resulting parameterized kernel is $2\pi$-periodic.
Discretizing this by $N$-point trapezoidal quadrature with some local corrected kernel $\tilde{G}$ as
\begin{align}
  S_k \varphi(x_i) \approx \sum_{j = 1}^{N} \tilde{G}_k(x_i, y_j) \, \varphi(y_j) \, |y^{\prime}_j| \frac{2\pi}{N},
\end{align}
where for $j = 1, 2, \ldots, N$,
$\theta_j$ are equispaced points in the parameter space
 and $y_j = y(\theta_j)$ are the corresponding quadrature points in the original coordinates.
Similarly, $x_j = x(\theta_j)$ is the observation point for $i = 1, 2, \ldots, N$.
We define the discretized single-layer potential operator $\bm{S}_k^\prime$ whose $i, j$ component is expressed as
\begin{equation}
  [\bm{S}_k^\prime ]_{ij} = \tilde{G}_k(x_i, y_j) \, w_j ,
\end{equation}
where we refer to $w_j = |y^{\prime}_j| \frac{2\pi}{N} \, (= |x^{\prime}_j| \frac{2\pi}{N})$ as the weight of the quadrature rule.
We can interpret matrix $\bm{S}_k^\prime$ as being scaled by $w_j$ for the $j$-th column.
Then, we represent a row-scaled (discretized) single-layer potential operator $\bm{S}_k$ by
\begin{equation}
  [\bm{S}_k ]_{ij} = w_i \tilde{G}_k(x_i, y_j) \, w_j. \label{eq:scaled_slp}
\end{equation}
If a local corrected kernel $\tilde{G}_k$ is symmetric,
then we can conclude that matrix $\bm{S}_k$ is symmetric (but is not Hermitian).
However, in general, the local corrected part of $\tilde{G}$ (the blue band in Figure \ref{fig:local_band}) is not symmetric.
Therefore, we restrict ourselves to the quadrature methods that preserve symmetry.
In what follows, we only discuss two-dimensional problems for simplicity.
Although the proposed framework is applicable to three-dimensional problems
via the one-point corrected quadrature for a torus geometry \cite{wu2021corrected},
the construction of a quadrature method preserving symmetry for any geometry is not easy.
\begin{figure}[tb]
  \centering
  \hfill
  \includegraphics[width=0.20\linewidth]{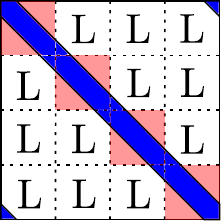}
  \hfill
  \includegraphics[width=0.20\linewidth]{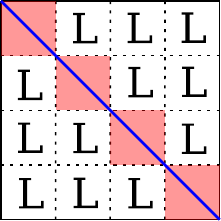}
  \hfill
  \caption{
    Illustrations of matrices with a typical weakly admissible condition for one-dimensional data,
    where all off-diagonal blocks are approximated as low-rank.
    The left and right illustrations correspond to a high-order and the lowest-order discretizations, respectively.
    The local corrected part is indicated by the blue band.
    The red blocks represent the dense parts, whereas blocks labeled ``L'' represent the low-rank approximable blocks.
  }
  \label{fig:local_band}
\end{figure}

In the method of the present study, we only correct the diagonal term of $\bm{S}_k$ using
the one-point zeta-corrected quadrature \cite{wu2021zeta}.
This one-point correction is consistent with the fact that the hierarchical semiseparable representation (used in Section \ref{sec:fds}) may not be fully suitable for high-order discretizations \cite{matsumoto2026proxy}.
Herein, we refer to the one-point zeta-corrected quadrature as the lowest-order zeta-corrected quadrature that uses the correction term only on the observation node.
We can define the one-point corrected $\tilde{G}_k^S$ (for the single-layer operator) by
\begin{equation}
  \tilde{G}_k^{S}(x_i, y_j) =
  \begin{dcases}
    \text{Corrected term for $S_k$} & (i = j), \\
    G_k(x_i, y_j) & (i \neq j),
  \end{dcases}
\end{equation}
so that $\tilde{G}_k^{S} (x_i, y_j)$ ($i, j = 1, 2, \ldots, N$) satisfies
\begin{equation}
  \qty| \int_\Gamma G_k(x_i, y) \varphi (y) \dd s(y) - \sum_{j = 1}^{N} \tilde{G}_k^{S}(x_i, y_j) \varphi(y_j) w_j | =  O(h^2), \label{eq:zeta_s}
\end{equation}
where $h$ is the interval of the equispaced quadrature points in the parameter space
and \eqref{eq:zeta_s} is point-wise convergent as $N \to \infty$ \cite{wu2021zeta}.
Here, ``point-wise'' means the dependency on $\varphi$.
Using this corrected kernel $\tilde{G}_k^{S}$,
 the row-scaled discretized single-layer potential operator $\bm{S}_k$ is redefined as
\begin{equation}
  [\bm{S}_k ]_{ij} = w_i \tilde{G}_k^{S_k}(x_i, y_j) w_j. \label{eq:scaled_slp}
\end{equation}
Relation \eqref{eq:zeta_s} can be rewritten using a discretized vector $\bm{\varphi}$,
 whose $i$-th component is expressed as $\qty[\bm{\varphi}]_i = \varphi(x_i)$,
 as
\begin{equation}
  \qty| w_i S_k \varphi (x_i) - \qty[\bm{S}_k \bm{\varphi}]_i | =  O(h^2), \quad i = 1, 2, \ldots, N.
\end{equation}
It is also point-wise convergent since $w_i$ is bounded from the $C^2$ smoothness of $\Gamma$.
Similarly to $\bm{S}_k$, we define row-scaled discretized layer potential operators $\bm{D}_k$, $\bm{D}_k^*$, and $\bm{N}_k$ corresponding to $D_k$, $D_k^*$, and $N_k$ using the zeta-corrected quadrature by
\begin{align}
  [\bm{D}_k ]_{ij} = w_i \tilde{G}_k^{D_k}(x_i, y_j) w_j, \label{eq:scaled_dlp} \\
  [\bm{D}_k^* ]_{ij} = w_i \tilde{G}_k^{D_k^*}(x_i, y_j) w_j, \label{eq:scaled_d_slp} \\
  [\bm{N}_k ]_{ij} = w_i \tilde{G}_k^{N_k}(x_i, y_j) w_j, \label{eq:scaled_d_dlp}
\end{align}
where each corrected kernel is given by
\begin{align}
  \tilde{G}_k^{{D}}(x_i, y_j) &=
  \begin{dcases}
    \text{Corrected term for $D_k$} & (i = j), \\
    \pdv{G_k(x_i, y_j)}{\nu (y_j)} & (i \neq j),
  \end{dcases}
  \\
  \tilde{G}_k^{{D}^*}(x_i, y_j) &=
  \begin{dcases}
    \text{Corrected term for ${D}_k^*$} & (i = j), \\
    \pdv{G_k(x_i, y_j)}{\nu (x_j)} & (i \neq j),
  \end{dcases}
  \\
  \tilde{G}_k^{{N}}(x_i, y_j) &=
  \begin{dcases}
    \text{Corrected term for ${N_k}$} & (i = j), \\
    \pdv{G_k(x_i, y_j)}{\nu (x_j)}{\nu (y_j)} & (i \neq j),
  \end{dcases}
\end{align}
for $i, j = 1, 2, \ldots N$.
These discretized potential operators $\bm{D}_k$, $\bm{D}_k^*$, and $\bm{N}_k$ satisfy
\begin{align}
  \qty| w_i D_k \varphi (x_i) - [\bm{D}_k \bm{\varphi}]_i | &= O(h^2), \\
  \qty| w_i D_k^* \varphi (x_i) - [\bm{D}_k^* \bm{\varphi}]_i | &= O(h^2), \\
  \qty| w_i N_k \varphi (x_i) - [\bm{N}_k \bm{\varphi}]_i | &= O(h). \label{eq:estimate_N}
\end{align}
In this row-scaled formulation, we have the symmetric relations
\begin{align}
  \bm{S}_k = \bm{S}_k^\top, \quad \bm{D}_k = {\bm{D}_k^*}^\top, \quad \bm{N}_k = \bm{N}_k^\top. \label{eq:syms}
\end{align}
For a detailed discussion of the correction terms and their error estimates,
we refer the reader to \cite{wu2021zeta, wu2023unified}.

\begin{example}[Numerical symmetry] \label{ex:sym}
  \begin{figure}[tbp]
    \centering
    \includegraphics[width = 0.23\linewidth]{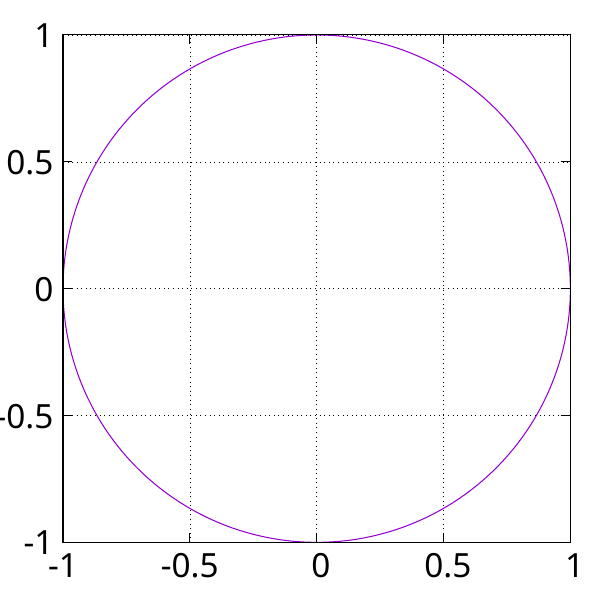}
    \hfill
    \includegraphics[width = 0.23\linewidth]{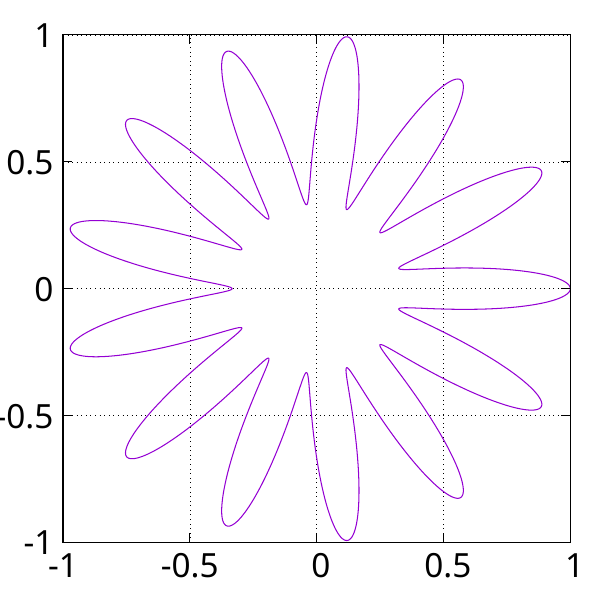}
    \hfill
    \includegraphics[width = 0.23\linewidth]{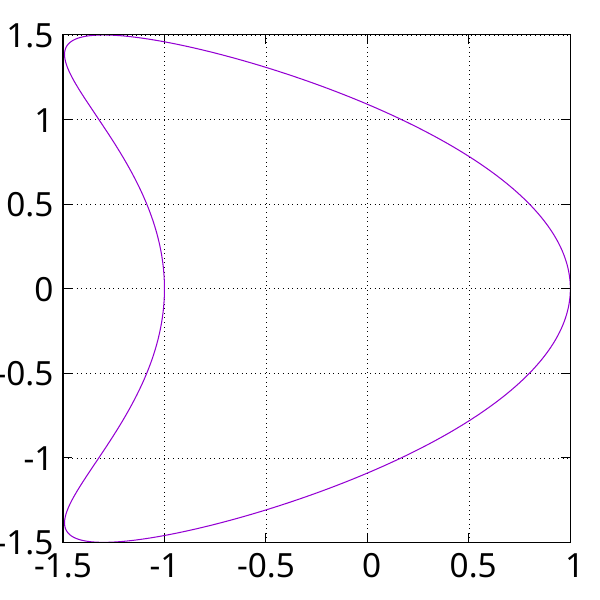}
    \hfill
    \includegraphics[width = 0.23\linewidth]{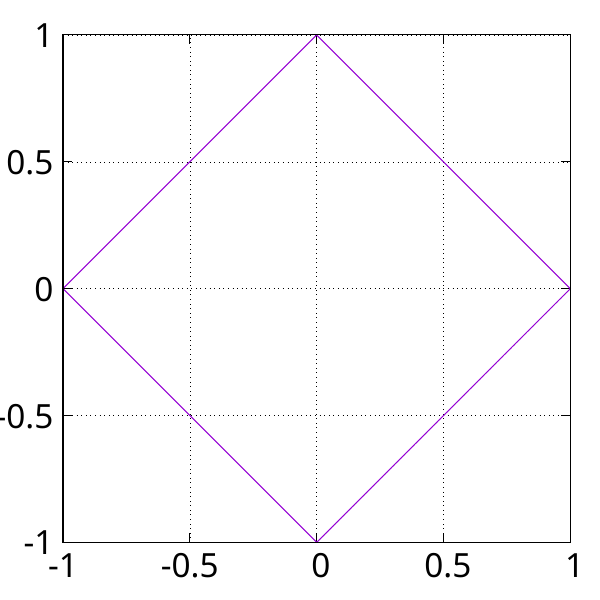}
    \caption{Geometries used in numerical examples. From left to right, these are the unit circle, smooth star, kite, and square.}
    \label{fig:geometries}
  \end{figure}
  We demonstrate a simple numerical verification of the symmetric relations \eqref{eq:syms}.
  Let $\bm{Z}$ and $\bm{Z}^\prime$ be a row-scaled discretized potential operator and its non-scaled counterpart,
  where $\bm{Z}$ stands for $\bm{S}_k$, $\bm{D}_k$, $\bm{D}_k^*$, or $\bm{N}_k$.
  We use the unit circle and smooth star illustrated in Figure \ref{fig:geometries} as the geometries.
  The number of quadrature points is 3000.
  The wavenumber $k$ and $\tilde{k}$ are 3 and 6, respectively.
  We summarize the results of computing $\norm{\bm{S}_k - \bm{S}_k^\top}_F$ and other relations using the scaled discretized operators in Table \ref{tab:symmetric_check_scaled},
  where $\norm{\cdot}_F$ is the Frobenius norm.
  Table \ref{tab:symmetric_check_prime} using the non-scaled discretized operators is a counterpart of Table \ref{tab:symmetric_check_scaled}.
  In these tables, $m$ is the number of the local corrected points for the zeta-corrected quadrature.
  From Table \ref{tab:symmetric_check_scaled},
  we can see that there is no noticeable difference for the unit circle for $\bm{S}_k$, $\bm{D}_k$, $\bm{D}_k^*$, and $\bm{N}_k$.
  However, the use of the row-scaled $\bm{N}_k$ combined with the one-point correction ($m = 1$) is essential for the smooth star, which lacks local symmetry.
  The results in Table \ref{tab:symmetric_check_prime} confirm that local symmetry of the boundary shape is directly linked to the symmetry of the coefficient matrix generated by the standard Nystr\"om method.
  \begin{table}[htbp]
    \centering
    \caption{Numerical symmetries of row-scaled discretized operators \eqref{eq:scaled_slp}, \eqref{eq:scaled_dlp}, \eqref{eq:scaled_d_slp}, and \eqref{eq:scaled_d_dlp}.}
    \label{tab:symmetric_check_scaled}
    \resizebox{1.0\linewidth}{!}{
      \begin{tabular}{llcccc}
        \toprule
        Geometry & $m$ & $\norm{\bm{S}_k - \bm{S}_k^\top}_F$ & $\norm{\bm{D}_k - {\bm{D}_k^*}^\top}_F$ &  $\norm{\bm{N}_k - \bm{N}_k^\top}_F$ & $\norm{(\bm{N}_k - \bm{N}_{\tilde{k}}) - (\bm{N}_k - \bm{N}_{\tilde{k}})^\top}_F$ \\
        \midrule
        Unit circle & 1  & $4.74 \times 10^{-17}$ & $1.05 \times 10^{-16}$ & $1.44 \times 10^{-17}$ & $2.02 \times 10^{-17}$ \\
        & 41 & $4.74 \times 10^{-17}$ & $1.05 \times 10^{-16}$ & $4.52 \times 10^{-17}$ & $3.88 \times 10^{-17}$ \\
        \midrule
        Smooth star & 1  & $1.19 \times 10^{-16}$ & $1.67 \times 10^{-16}$ & $5.22 \times 10^{-17}$ & $6.13 \times 10^{-17}$ \\
        & 41 & $1.20 \times 10^{-16}$ & $1.67 \times 10^{-16}$ & $2.21 \times 10^{-3}$ & $8.58 \times 10^{-17}$  \\
        \bottomrule
      \end{tabular}
    }
  \end{table}
  \begin{table}[htbp]
    \centering
    \caption{Numerical symmetries of non-scaled discretized operators as references.}
    \label{tab:symmetric_check_prime}
    \resizebox{1.0\linewidth}{!}{
      \begin{tabular}{llcccc}
        \toprule
        Geometry    & $m$ & $\norm{\bm{S}_k^\prime - {\bm{S}_k^\prime}^\top}_F$  & $\norm{\bm{D}_k^\prime - {{\bm{D}_k^*}^\prime}^\top}_F$ & $\norm{\bm{N}_k^\prime - {\bm{N}_k^\prime}^\top}_F$ & $\norm{(\bm{N}_k^\prime - \bm{N}_{\tilde{k}}^\prime) - (\bm{N}_k^\prime - \bm{N}_{\tilde{k}}^\prime)^\top}_F$ \\
        \midrule
        Unit circle & 1   & $1.28 \times 10^{-16}$                            & $1.80 \times 10^{-16}$ & $5.19 \times 10^{-17}$ & $2.02 \times 10^{-17}$ \\
                    & 41  & $1.28 \times 10^{-16}$                            & $1.80 \times 10^{-16}$ & $8.93 \times 10^{-17}$ & $3.88 \times 10^{-17}$ \\
        \midrule
        Smooth star & 1   & $5.31 \times 10^{-1}$                             & $5.17 \times 10^{-1}$  & $5.26 \times 10^{-2}$ & $1.87 \times 10^{-3}$ \\
                    & 41  & $5.31 \times 10^{-1}$                             & $5.17 \times 10^{-1}$  & $5.84 \times 10^{-2}$ & $1.20 \times 10^{-3}$ \\
        \bottomrule
      \end{tabular}
    }
  \end{table}

  On the other hand, Table \ref{tab:symmetric_check_scaled} indicates that for $N_k - N_{\tilde{k}}$, which exhibits at most a weak singularity, symmetry can be preserved through row scaling even with $m = 41$.
  This suggests that solvers based on strong admissibility can leverage symmetry to enhance the efficiency of low-rank approximation.
  Indeed, the symmetric M\"uller-like boundary integral equation presented in Section \ref{sec:bie} involves an operator containing $N_k - N_{\tilde{k}}$.
  However, since the development of a solver based on strong admissibility is beyond the scope of the present paper, we limit ourselves to suggesting its possibility.
\end{example}

\begin{remark}
If we allow ambiguous descriptions of function spaces of functions $\varphi, \psi : \Gamma \to \mathbb{C}$, the bilinear form,
\begin{equation}
  \bilinear*{ \varphi , \psi } = \int_\Gamma \varphi (x) \psi (x) \dd s(x),
\end{equation}
gives the following adjoint relations \cite{colton2013inverse}:
\begin{equation}
  \bilinear{S_k \varphi, \psi} = \bilinear{\varphi, S_k \psi},
  \quad \bilinear{D_k \varphi, \psi} = \bilinear{\varphi, D_k^* \psi},
  \quad \bilinear{N_k \varphi, \psi} = \bilinear{\varphi, N_k \psi}.
\end{equation}
The symmetric Galerkin boundary element method formulation can be constructed using appropriate basis and test functions owing to these adjoint relations.
In fact, the Galerkin method is used in the extension to the elastic scattering problem in Section \ref{sec:extension} for accelerating the low-rank approximation and the direct solver using the symmetric formulation similarly to the Helmholtz counterpart.
\end{remark}

\subsection{Boundary integral equations} \label{sec:bie}
The discussion in the previous subsection demonstrated that the discretization matrices associated with layer potentials can be symmetrized. However, this approach requires a one-point correction, which has the drawback of hindering the use of high-order quadrature rules.
To address this, we propose a Müller-like boundary integral equation that maintains symmetry while having a singularity that is at most weakly singular, in an attempt to recover the order of convergence.

We define the diagonal weight matrix $W$ by
\begin{equation}
W = \diag(w_1, w_2, \ldots, w_N).
\end{equation}
Clearly, $W$ is invertible.
For transmission problem \eqref{eq:problem}, it is well known that the PMCHWT boundary integral equation \cite{POGGIO1973159, Chang1977surface} uniquely solves the problem and that its solution depends continuously on perturbations of the right-hand side \cite{hiptmair2022spurious}.
The standard PMCHWT formulation in Nystr\"om discretized form is expressed as
\begin{equation}
  \mqty[
  -\qty(\varepsilon_{0} \bm{S}_{k_{0}}^\prime + \varepsilon_{1} \bm{S}_{k_1}^\prime) & \bm{D}_{k_0}^\prime + \bm{D}_{k_1}^\prime \\
  {\bm{D}_{ k_0}^{\ast \prime}} + {\bm{D}_{k_1}^{\ast \prime}} & -\qty(\frac{1}{\varepsilon_{0}} \bm{N}_{k_0}^\prime + \frac{1}{\varepsilon_{1}} \bm{N}_{k_1}^\prime)
  ]
  \mqty[
  \bm{q} \\
  \bm{u}
  ]
  =
  \mqty[
  -\bm{u}^{\mathrm{in}} \\
  \frac{1}{\varepsilon_{0}} \bm{q}^{\mathrm{in}}
  ], \label{eq:naive_pmchwt}
\end{equation}
where $\bm{u}$, $\bm{q}$, $\bm{u}^{\mathrm{in}}$, and $\bm{q}^{\mathrm{in}}$ are the discretized $u, q$ of \eqref{eq:problem},
 the discretized incident wave on $\Gamma$, and that latter's normal derivative, respectively.
Recalling symmetric relations \eqref{eq:syms} and multiplying the first and second rows of this equation by $W$, we obtain the symmetric PMCHWT boundary integral equation as
\begin{align}
  \mqty[
  W & \\
  & W
  ]
  \mqty[
  -\qty(\varepsilon_{0} \bm{S}_{k_{0}}^\prime + \varepsilon_{1} \bm{S}_{k_1}^\prime) & \bm{D}_{k_0}^\prime + \bm{D}_{k_1}^\prime \\
  {\bm{D}_{ k_0}^{\ast \prime}} + {\bm{D}_{k_1}^{\ast \prime}} & -\qty(\frac{1}{\varepsilon_{0}} \bm{N}_{k_0}^\prime + \frac{1}{\varepsilon_{1}} \bm{N}_{k_1}^\prime)
  ]
  \mqty[
  \bm{q} \\
  \bm{u} 
  ]
  &=
  \mqty[
  W & \\
  & W
  ]
  \mqty[
  -\bm{u}^{\mathrm{in}} \\
  \frac{1}{\varepsilon_{0}} \bm{q}^{\mathrm{in}}
  ] \nonumber \\
  \mqty[
  -\qty(\varepsilon_{0} \bm{S}_{k_{0}} + \varepsilon_{1} \bm{S}_{k_1}) & \bm{D}_{k_0} + \bm{D}_{k_1} \\
  {\bm{D}_{ k_0}^{\ast}} + {\bm{D}_{k_1}^{\ast}} & -\qty(\frac{1}{\varepsilon_{0}} \bm{N}_{k_0} + \frac{1}{\varepsilon_{1}} \bm{N}_{k_1})
  ]
  \mqty[
  \bm{q} \\
  \bm{u} 
  ]
  &=
  \mqty[
  -W \bm{u}^{\mathrm{in}} \\
  \frac{1}{\varepsilon_{0}} W \bm{q}^{\mathrm{in}}
  ]. \label{eq:symmetric_pmchwt}
\end{align}
Since $\diag \{W, W \}$ is also invertible,
the coefficient matrix of \eqref{eq:symmetric_pmchwt} is invertible if the coefficient matrix of \eqref{eq:naive_pmchwt} is invertible.
As will be demonstrated in Section \ref{sec:numerical_demo},
 an efficient fast direct solver can be constructed based on this symmetric PMCHWT formulation,
 but its convergence rate is $O(h)$.
This rate coincides with error estimate \eqref{eq:estimate_N} with respect to $N_k$.

Therefore, we propose the following M\"uller-like boundary integral equation, in which the integral operator is at most weakly singular while also possessing symmetry:
\begin{multline}
  \mqty[
  -\qty(\varepsilon_{0}^{2} \bm{S}_{k_{0}} - \varepsilon_{1}^2 \bm{S}_{k_1})
  &
  \varepsilon_{0} \qty(\bm{D}_{k_0} - \frac{1}{2}W)- \varepsilon_{1} \qty(\bm{D}_{k_1} + \frac{1}{2}W)
  \\
  \varepsilon_{0} \qty(\bm{D}_{k_0}^{\ast} + \frac{1}{2}W)
  - \varepsilon_{1} \qty(\bm{D}_{k_1}^{\ast} - \frac{1}{2}W)
  &
  -\qty(\bm{N}_{k_0} - \bm{N}_{k_1})
  ]
  \mqty[
  \bm{q} \\
  \bm{u}
  ] \\
  =
  \mqty[
  -\varepsilon_0 W \bm{u}^{\mathrm{in}} \\
  \frac{1}{\varepsilon_{0}} W \bm{q}^{\mathrm{in}}
  ], \label{eq:symmetric_muller}
\end{multline}
where $\frac{1}{2}W$ is the term arising from the free terms.
This equation is derived by rearranging the layer potentials of ``M\"uller3'' from \cite{misawa2012jascome},
scaling them by material constants and $W$,
and multiplying one of the rows by $-1$.
It should be noted that \eqref{eq:symmetric_muller} is not a Fredholm integral equation of the second kind, as its off-diagonal parts contain free terms.
Consequently, we refer to this formulation as ``M\"uller-like''.
Furthermore, we emphasize that \eqref{eq:symmetric_muller} maintains symmetry, except for these aforementioned free terms.
Therefore, \eqref{eq:symmetric_muller} is well suited for fast direct solvers, though it is less favorable for iterative solvers.
In fact, we will demonstrate in Section \ref{sec:numerical_demo} that a fast direct solver based on \eqref{eq:symmetric_muller} achieves a convergence rate of $O(h^3)$.

\section{Low-rank approximation exploiting symmetry in a fast direct solver} \label{sec:fds}
By observing the procedure of the proxy method \cite{MARTINSSON20051},
it becomes immediately apparent that the computational cost of the low-rank approximation can be reduced
if the boundary integral equations exhibit symmetry.
The computational cost for constructing low-rank approximations
is dominant in a fast direct solver for boundary integral equations using the proxy method \cite{matsumoto2026fast}.
In this section, we focus only on symmetric M\"uller-like formulation \eqref{eq:symmetric_muller}, but the same arguments apply to symmetric PMCHWT formulation \eqref{eq:symmetric_pmchwt}.

\subsection{Block formulation and low-rank approximation}
Assuming that there are $N$ quadrature points on $\Gamma$,
 the coefficient matrix in \eqref{eq:symmetric_muller} has dimensions $2N \times 2N$.
By assigning an index to each quadrature point, the entire index set $J_0 = \{1, 2, \ldots, N \}$ can be partitioned into $p$ subsets $J_1, J_2, \ldots, J_p$ such that $J_0 = \bigcup_{i = 1}^p J_i$.
For each $i = 1, 2, \ldots, p$, index set $J_i$ corresponds to the subinterval $\Gamma_i$ of $\Gamma$.
For simplicity, we assume that each subset has the same size $n$, which yields $N = pn$.
The linear algebraic equation arising from \eqref{eq:symmetric_muller} can be transformed into the following block forms by an appropriate reordering:
\begin{equation}
  \mqty[
  A_{11} & A_{12} & \cdots & A_{1p} \\
  A_{21} & A_{22} & \cdots & A_{2p} \\
  \vdots & \vdots &\ddots &\vdots \\
  A_{p1} & A_{p2} & \cdots & A_{pp}
  ]
  \mqty[
  x_1 \\
  x_2 \\
  \vdots \\
  x_p
  ]
  =
  \mqty[
  f_1 \\
  f_2 \\
  \vdots \\
  f_p
  ],
  \label{eq:block_linear}
\end{equation}
where subscripts $i$ and $j$ denote the blocks corresponding to index subsets $J_i$ and $J_j$, respectively.
For $i, j = 1, 2, \ldots, p$, block $A_{ij}$ has a $2 \times 2$ structure expressed as
\begin{multline}
  A_{ij} = \mqty[
    A_{ij}^{(11)} & A_{ij}^{(12)} \\
    A_{ij}^{(21)} & A_{ij}^{(22)}
  ]
  =
  \\
  \mqty[
  -\qty(\varepsilon_{0}^{2} \bm{S}_{k_{0}}^{ij} - \varepsilon_{1}^2 \bm{S}_{k_1}^{ij})
  &
  \varepsilon_{0} \qty(\bm{D}_{k_0}^{ij} - \delta_{ij}\frac{W}{2})- \varepsilon_{1} \qty(\bm{D}_{k_1}^{ij} + \delta_{ij}\frac{W}{2})
  \\
  \varepsilon_{0} \qty({\bm{D}_{k_0}^{\ast ij}} + \delta_{ij}\frac{W}{2})
  - \varepsilon_{1} \qty({\bm{D}_{k_1}^{\ast ij}} - \delta_{ij}\frac{W}{2})
  &
  -\qty(\bm{N}_{k_0}^{ij} - \bm{N}_{k_1}^{ij})
  ],
  \label{eq:block_A}
\end{multline}
where $\delta_{ij}$ is the Kronecker delta.
Here, superscript $ij$ of the discretized layer operators $\bm{S}_{k}, \bm{D}_{k}, \bm{D}_{k}^*,$ and $\bm{N}_{k}$ 
denotes the submatrix corresponding to $i$-th row and $j$-th column index sets $J_i$ and $J_j$, respectively.
In \eqref{eq:block_A}, the subblocks of $A_{ij}$ are denoted by $A_{ij}^{(11)}$, $A_{ij}^{(12)}$, $A_{ij}^{(21)}$, and $A_{ij}^{(22)}$, corresponding to the top-left, top-right, bottom-left, and bottom-right components, respectively.
In \eqref{eq:block_linear}, $x_i$ and $f_i$ are the subvectors of the solution and the right-hand side of \eqref{eq:symmetric_muller} associated with $J_i$, respectively.

In two-dimensional boundary integral equations, since matrix components with similar indices can often be arranged to be spatially close, many off-diagonal blocks of \eqref{eq:block_linear} admit  low-rank approximations \cite{ma2022scalable}.
For the present study, we assumed that all off-diagonal blocks of \eqref{eq:block_linear} are low-rank approximable,
which means that \eqref{eq:block_linear} has the structure of the hierarchically semiseparable (HSS) representation \cite{chandrasekaran2005Calcolo}.
Therefore, the $2 \times 2$ block in the off-diagonal part $A_{ij}$, for $i, j = 1, 2, \ldots, p$ with $i \neq j$, can be low-rank approximated as
\begin{equation}
  A_{ij}
  \approx L_{i} S_{ij} R_{j} =
  \mqty[L_{i}^{(1)} & \\ & L_{i}^{(2)}]
  \mqty[
    A_{\tilde{i}\tilde{j}}^{(11)} & A_{\tilde{i}\tilde{j}}^{(12)} \\
  A_{\tilde{i}\tilde{j}}^{(21)} & A_{\tilde{i}\tilde{j}}^{(22)}
  ]
  \mqty[R_{j}^{(1)} & \\ & R_{j}^{(2)}]
  \label{eq:low-rank}
\end{equation}
where 
$S_{ij} \in \mathbb{C}^{2 k \times 2 k}$ with $k \ll n$
is the so-called skeleton matrix of $A_{ij}$, and
$L_{i} \in \mathbb{C}^{2 n \times 2k}$ and $R_{j} \in \mathbb{C}^{2 k \times 2 n}$
are the left and right shared coefficients  for linear combinations, respectively.
The submatrices $L_{i}^{(1)} \in \mathbb{C}^{n \times k}$ and $L_{i}^{(2)} \in \mathbb{C}^{n \times k}$ interpolate the first and second block rows of $S_{ij}$.
Similarly, the submatrices $R_{i}^{(1)} \in \mathbb{C}^{k \times n}$ and $R_{i}^{(2)} \in \mathbb{C}^{k \times n}$ interpolate the first and second block columns of $S_{ij}$.
Here, subscripts $\tilde{i}\tilde{j}$ of $A_{\tilde{i}\tilde{j}}^{(pq)}$ ($p, q = 1, 2$) are the skeleton indices selected from the $i$-th and $j$-th index sets using the proxy method \cite{MARTINSSON20051} and the double-sided interpolative decomposition \cite{cheng2005compression}.

\begin{figure}[tbp]
  \begin{minipage}[c]{0.54\linewidth}
    \centering
    \includegraphics[width = 0.4\linewidth]{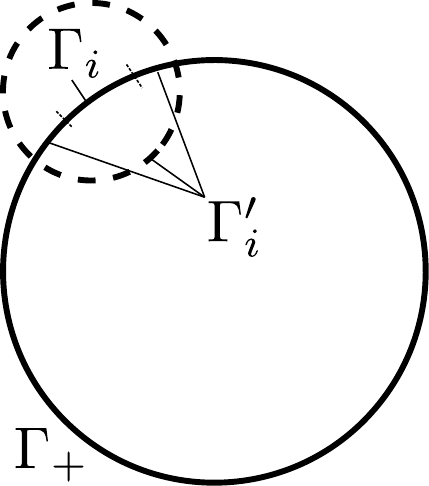}
    \caption{Local virtual boundary in proxy method.}
    \label{fig:proxy}
  \end{minipage}
  \hfill
  \begin{minipage}[c]{0.44\linewidth}
    \centering
    \captionof{table}{Relationship between $\Gamma_{i}$ and $\Gamma_{i}^\prime$.}
    \label{tab:observation}
            {\footnotesize
              \begin{tabular}{lll}
                \toprule
                & For $L_i$              & For $R_i$                 \\
                \midrule
                Observations & $\Gamma_{i}$        & $\Gamma_{i}^{\prime}$ \\
                Sources      & $\Gamma_{i}^{\prime}$ & $\Gamma_{i}$          \\
                \bottomrule
              \end{tabular}
            }
  \end{minipage}
\end{figure}
In the standard proxy method, a local virtual boundary is introduced to enclose subset $\Gamma_i \subset \Gamma$ associated with index set $J_i$, as illustrated in Figure \ref{fig:proxy}.
We denote by $\Gamma_i^\prime$ the union of this local virtual boundary and a neighborhood of $\Gamma_i$.
According to potential theory, the interaction between $\Gamma_i$ and the remainder of the boundary, $\Gamma_+ = \Gamma \setminus \Gamma_i$, can be replaced by the interaction between $\Gamma_i$ and $\Gamma_i^\prime$ \cite{MARTINSSON20051}.
We denote these interaction matrices by $M_{R_i}$ and $M_{L_i}$, defined as
\begin{align}
  M_{R_i} = \mqty[
    A_{i^\prime j}^{(11)} & A_{i^\prime j}^{(12)} \\
    A_{i^\prime j}^{(21)} & A_{i^\prime j}^{(22)}
  ], \quad
  M_{L_i} = \mqty[
    A_{ij^\prime}^{(11)} & A_{ij^\prime}^{(12)} \\
    A_{ij^\prime}^{(21)} & A_{ij^\prime}^{(22)}
  ],
\end{align}
where the subscripts $i^\prime j$ and $ij^\prime$ represent the interactions between $\Gamma_i^\prime$ and $\Gamma_i$, and between $\Gamma_i$ and $\Gamma_i^\prime$, respectively, as listed in Table \ref{tab:observation}.
We obtain $L_i^{(1)}$ and $L_i^{(2)}$ by factorizing respectively the upper half and lower half of $M_{L_i}$ using the interpolative decomposition.
Similarly, $R_i^{(1)}$ and $R_i^{(2)}$ are obtained by factorizing respectively the left half and right half of $M_{R_i}$ using the interpolative decomposition.

We note that the proxy method can be accelerated by simply setting $L_i = R_i^\top$, provided that the condition $[M_{L_i}]^\top = M_{R_i}$ is satisfied.
More precisely, $A_{ij}$ can be low-rank approximated as
\begin{equation}
  A_{ij}
  \approx R_{i}^\top S_{ij} R_{j} =
  \mqty[ [R_{i}^{(1)}]^\top & \\ & [R_{i}^{(2)}]^\top]
  \mqty[
    A_{\tilde{i}\tilde{j}}^{(11)} & A_{\tilde{i}\tilde{j}}^{(12)} \\
  A_{\tilde{i}\tilde{j}}^{(21)} & A_{\tilde{i}\tilde{j}}^{(22)}
  ]
  \mqty[R_{j}^{(1)} & \\ & R_{j}^{(2)}].
  \label{eq:low-rank_R}
\end{equation}
This approach requires computing only $M_{R_i}$, and reduces the number of factorizations by half in the proxy method.
In the proposed procedure, since both the discretization scheme described in Section \ref{sec:nystrom} and the boundary integral equations in \eqref{eq:symmetric_muller} possess symmetry, the acceleration using \eqref{eq:low-rank_R} can be readily implemented.

\begin{remark}
There exists a method for enforcing $L_i = R_i^\top$ by factorizing the augmented matrix $[ M_{L_i} \mid M_{R_i}^\top ]$ in non-symmetric formulations \cite[Remark 8]{cheng2005compression}. However, that approach is not suitable for accelerating the proxy method, as it requires computing both $M_{L_i}$ and $M_{R_i}$ and factorizing the double-sized matrix.
Moreover, a disadvantage regarding the low-rank approximation accuracy has been reported \cite{cheng2005compression}.
\end{remark}

\begin{remark}
The standard double-sided interpolative decomposition for a complex matrix requires that not only the transpose be taken but also the complex conjugate, when computing left coefficient $L_i$ \cite{martinsson2019book}.
We note that taking the complex conjugate is not essential for the double-sided interpolative decomposition.
\end{remark}

\subsection{Compression technique for linear algebraic equations}
To solve \eqref{eq:block_linear} efficiently, we apply a multilevel compression technique for linear algebraic equations.
This fast direct solver can be interpreted as a variant of the Martinsson--Rokhlin solver \cite{MARTINSSON20051} or the HSS-ULV decomposition \cite{chandrasekaran2006fast}.
Since the algorithm is almost identical to those in prior studies \cite{MATSUMOTO2025106148, matsumoto2025efficient},
 except for the improved efficiency of the low-rank approximation, we keep its description to a minimum.

Further, since all off-diagonal blocks in \eqref{eq:block_linear} are low-rank approximated using the shared coefficient $R_i$
as in \eqref{eq:low-rank_R}, we can compress the linear algebraic equations block-row-wise.
For $i = 1, 2, \ldots, p$, the $i$-th row-block equation is expressed as
\begin{align}
  A_{ii} x_{i} + \sum_{j = 1, j \neq i}^p R_i^\top S_{ij} R_{j} x_{j} = f_i.
\end{align}
Using $y_i = R_i x_i$ and $\tilde{A}_i = [R_i A_{ii}^{-1} R_i^\top]^{-1}$, this equation is compressed into
\begin{align}
  \tilde{A}_{i} y_{i} + \sum_{j = 1, j \neq i}^p S_{ij} y_{j} = \tilde{A}_{i} R_i A_{ii}^{-1} f_i.
  \label{eq:compressed_eq}
\end{align}
This transformation results in the compressed system, whose coefficient matrix has dimensions $2kp \times 2kp$, expressed as
\begin{align}
  \mqty[
    \tilde{A}_{1} & S_{12} & \cdots & S_{1p} \\
    S_{21} & \tilde{A}_{2} & \cdots & S_{1p} \\
    \vdots & \vdots & \ddots & \vdots \\
    S_{p1} & S_{p2} & \cdots & \tilde{A}_{p} \\
  ]
  \mqty[
    y_1 \\
    y_2 \\
    \vdots \\
    y_p
  ]
  =
  \mqty[
    \tilde{A}_{1} R_1 A_{11}^{-1} f_1 \\
    \tilde{A}_{2} R_2 A_{22}^{-1} f_2 \\
    \vdots \\
    \tilde{A}_{p} R_p A_{pp}^{-1} f_p
  ].
\end{align}
This system can be solved more efficiently than original system \eqref{eq:block_linear},
which has a coefficient matrix of dimensions $2N \times 2N$ (where $N = np$).
Compressed solution $y_i$ can be converted into original solution $x_i$ using the following block-row-wise relation:
\begin{align}
  x_i = A_{ii}^{-1} f_i - A_{ii}^{-1} R_i^{T} (\tilde{A}_{i} R_i A_{ii}^{-1} f_i - \tilde{A}_{i} y_i),
\end{align}
for $i = 1, 2, \ldots, p$.
Although this procedure is a single-level algorithm, it can be extended to a multi-level one
by using that skeleton blocks $S_{ij}$ can be compressed recursively by virtue of the properties of the fundamental solution.
For details on the multi-level algorithm, see \cite{MATSUMOTO2025106148}.

\subsection{Rank scaling parameter in the multilevel algorithm}
In the preceding discussion, row and column blocks of size $2n$ were uniformly low-rank approximated with a fixed rank $2k$.
However, in wave scattering problems, higher levels of the tree structure require increasingly larger ranks to maintain the accuracy for the approximation \cite{rokhlin1990rapid}.
Although the rank can be dynamically determined for each row,
 it is preferable from the perspective of parallel efficiency that the matrix sizes at each level be predetermined using fixed ranks.
Therefore, from a practical viewpoint,
we can consider an implementation that gradually increases the rank so that the rank at a higher level is a fixed multiple, $a$, of the rank at the preceding lower level.
In this subsection, we investigate this scaling parameter $a$.
For simplicity, in the following theorem, the size $2N$ arising from the transmission problem is written as $N$.

\begin{theorem} \label{thm:order}
Consider a perfect binary tree of depth $L$ with its root at level $0$, which partitions a total of $N \in \mathbb{N}$ quadrature points.
Suppose that there exist $p = 2^L$ clusters at level $L$, each of size $n$, such that $N = np$.
At level $L$, all off-diagonal blocks of size $n$ are approximated with initial rank $k_L \in \mathbb{N}$, after which every two clusters are merged into a size of $2k_L$ at level $L-1$.
At any level $l$ ($0 < l < L$), each off-diagonal block of size $2a^{L-l-1}k_L$ is approximated by a rank $k_l = a^{L - l} k_L$ matrix,  where $a \ge 1$ denotes a scaling parameter.
Then, the number of floating-point operations for the overall multilevel fast direct solver presented in Section \ref{sec:fds} is bounded by
\begin{equation}
  \begin{dcases}
    O(N) & (1 \le a < 2^{1/3}), \\
    O(N \log_2 N) & (a = 2^{1/3}), \\
    O(N^{3 \log_2 a}) & (a > 2^{1/3}),
  \end{dcases}
\end{equation}
where $2^{1/3} = 1.25992\cdots$.
\end{theorem}

\begin{proof}
  For $a = 1$, the overall complexity of the solver is $O(N)$ \cite{MARTINSSON20051}.
Although the compression algorithm includes many operations, such as interaction calculations in the proxy method and the interpolative decomposition of the interaction matrices, the diagonal block inversion accounts for the dominant share of floating-point operations.
Moreover, the computational cost of converting the compressed solution back to the original solution is much smaller than that of the compression phase.
Therefore, the total number of floating-point operations is bounded by the following cost required for the inversion of the block-diagonal matrices from leaf level $L$ to the root:
\begin{equation}
O(n^3 2^L) + O \left( \sum_{l = 1}^{L-1} \left( 2a^{L-l-1}k_L \right)^3 2^l \right) + O\left( \left( 2a^{L-1} k_L \right)^3 \right), \label{eq:estimate_cost}
\end{equation}
where $2^l$ denotes the number of clusters, and the first, second, and third terms correspond to the leaf, intermediate, and root levels, respectively.
\begin{itemize}
\item \textbf{Leaf level:} We immediately observe that the cost at the leaf level is $O(N)$; this is because  $2^L = O(N)$ follows from $L = O(\log_2 N)$ and $n = O(1)$.
\item \textbf{Root level:} The third term of \eqref{eq:estimate_cost} can be evaluated as 
  where $k_L = O(1)$.
  Consequently, $a > 2^{1/3}$ implies a dominant complexity of $O(N^{3 \log_2 a})$, whereas the cost is bounded by $O(N)$ when $1 \le a \le 2^{1/3}$.
\item \textbf{Intermediate levels:}
  From the second term of \eqref{eq:estimate_cost},
  the cost for the inversion of diagonal blocks across all intermediate levels is given by
  \begin{equation}
    \begin{aligned}
    O \left( \sum_{l = 1}^{L - 1} 2^l \cdot a^{3(L - l - 1)} \right)
    &= O \left( a^{3(L - 1)} \sum_{l = 1}^{L - 1} (2a^{-3})^{l} \right) \\
    &= \begin{dcases}
      O(N) & (1 \le a < 2^{1/3}), \\
      O(N \log_2 N) & (a = 2^{1/3}), \\
      O(N^{3 \log_2 a}) & (a > 2^{1/3}).
      \end{dcases}
    \end{aligned}
  \end{equation}
\end{itemize}
  Combining these estimates yields the statement.
\end{proof}

\begin{remark}
Theorem \ref{thm:order} can be summarized as follows: for $1 \le a < 2^{1/3}$, the complexity is bounded by $O(N)$ at the leaf level; for $a > 2^{1/3}$, it is bounded by $O(N^{3 \log_2 a})$ at the root level; and for $a = 2^{1/3}$, each intermediate level yields the same computational load as the leaf level, resulting in an overall bound of $O(N \log_2 N)$.
\end{remark}

\section{Extension to elastic transmission problems} \label{sec:extension}
We extend the efficient low-rank approximation technique presented in Section \ref{sec:fds},
 which exploits symmetry, to two-dimensional elastic wave scattering problems with transmissive inclusions.
In elastic wave scattering as well, several corrected quadrature schemes have been developed
 for the Nystr\"om discretization \cite{dominguez2024nystrom, tong2007nystrom}.
However, their construction is generally more complex than that for Helmholtz scattering,
 due to the nature of the elastodynamic fundamental solution.
Therefore, the symmetric Galerkin boundary element method \cite{bonnet1998symmetric} is attractive.
In this discretization method, not only are the singularity of the (elastodynamic) fundamental solution mitigated through integration by parts using appropriate test functions, but the resulting coefficient matrix is symmetric.

The difference in formulation compared to the Helmholtz equation lies in the fact that we are dealing with a vector field rather than a scalar field.
Specifically, within domain $\Omega_m$ ($m = 0, 1$), we consider time-harmonic scattered wave $[v_m^{\mathrm{sc}}(x)]_i$.
This wave is associated with the $i$-th direction of Cartesian coordinate system $x = [x_1, x_2] \in \mathbb{R}^2$
and is governed by the Navier--Cauchy equations.
We assume that $\Omega_m$ is filled by a linearly elastic, homogeneous, and isotropic solid with density $\rho^{(m)}$ and Lam\'e constants $\lambda^{(m)}$ and $\mu^{(m)}$.
In $\Omega_0$, there is an incident wave $v^{\mathrm{in}}$ that solves the Navier--Cauchy equation in $\mathbb{R}^2$.
The elastic transmission problem is to find the solution $v_m^{\mathrm{sc}}(x)$ of the Navier--Cauchy equation in $\Omega_m$ as
\begin{align}
    \mu^{(m)} \Delta [v_{m}^{\mathrm{sc}}]_i (x) + (\lambda^{(m)} + \mu^{(m)}) \sum_{j=1}^{2} \pdv{[v_{m}^{\mathrm{sc}}(x)]_j}{x_j}{x_i} + \rho^{(m)} \omega^{2} [v_{m}^{\mathrm{sc}}]_i (x) = 0, \\
    \hfill m = 0, 1, \quad i = 1, 2, \quad x \in \Omega_{m}, \label{eq:navier}
\end{align}
with the transmission condition and the outgoing radiation condition for elastic waves \cite{eringen1975elastodynamics}.
Similarly to Helmholtz scattering, 
we define boundary unknowns $v = [v_1, v_2]$ by passing the limit of the total field into $\Gamma$,
where $v_i$ denotes the $i$-th directional component of $v$.
Let $t = [t_1, t_2]$ be the traction of $v$.
We note that in this elastic wave scattering, the boundary regularity can be relaxed to the Lipschitz condition due to the use of the Galerkin method \cite{costabel1988boundary}.

To solve the elastic wave scattering problem using the Galerkin boundary element method,
 boundary $\Gamma$ is approximated by a polygonal mesh.
As with the discretization in our prior work \cite{matsumoto2026fast},
each directional component of traction $t$ is approximated by piecewise constant (P0) shape functions.
Therefore, $\bm{t}$ is a complex-valued vector composed of the values of $t$ at the centroids of the line segment element.
Similarly, each directional component of displacement $v$ is approximated by piecewise linear (P1) shape functions,
which constructs a complex-valued vector $\bm{v}$ composed of the values of $v$ at the vertices of the polygonal mesh.

It is well known that in linear elasticity problems, particularly in three dimensions, reducing hypersingular kernels to weakly singular ones via the subtraction of layer potential operators is highly non-trivial.
Recently, the Müller boundary integral equation has been proposed for two-dimensional elastic transmission problems \cite{LeLouer2026muller}.
 However, since it requires a product of operators, it is not well suited for a direct solver.
Therefore, the present paper focuses on the PMCHWT equation in elastic transmission problems.

The PMCHWT-formulated (Galerkin discretized) boundary integral equation is expressed as
\begin{equation}
  \mqty[
  -\qty(\bm{U}^{(0)} + \bm{U}^{(1)})   & \bm{T}^{(0)} + \bm{T}^{(1)} \\
  \bm{T}^{*(0)} + \bm{T}^{* (1)} & -\qty(\bm{H}^{(0)} + \bm{H}^{(1)})
  ]
  \mqty[
  \bm{t} \\
  \bm{v} 
  ]
  =
  \mqty[
       {-\bm{v}^{\mathrm{in}}} \\
       {\bm{t}^{\mathrm{in}}} \\
       ],
       \quad x \in \Gamma, \label{eq:elastic_pmchwt}
\end{equation}
where $\bm{U}^{(m)}$ and $\bm{T}^{(m)}$ represent the single- and double-layer operators in elastodynamics associated with the parameters in $\Omega_m$, and $\bm{T}^{*(m)}$ and $\bm{H}^{(m)}$ are their tractions.
Here, we define $\bm{v}^{\mathrm{in}}$ as the vector obtained by testing $v^{\mathrm{in}}$ with the P0 shape functions, 
and $\bm{t}^{\mathrm{in}}$ as the vector obtained by testing $t^{\mathrm{in}}$ with the P1 shape functions, 
where $t^{\mathrm{in}}$ is the traction of $v^{\mathrm{in}}$.
In the Galerkin method, P0 functions are used as the test functions for the first row of \eqref{eq:elastic_pmchwt},
 whereas P1 functions are used for the second row.

Analogous to \eqref{eq:syms}, this discretization yields symmetric relations:
\begin{equation}
  \bm{U}^{(m)} = [\bm{U}^{(m)}]^\top, \quad \bm{T}^{(m)} = [\bm{T}^{* (m)}]^\top, \quad \bm{H}^{(m)} = [\bm{H}^{(m)}]^\top.
\end{equation}
Therefore, the coefficient of \eqref{eq:elastic_pmchwt} is symmetric.
Provided that one notes that it handles a vector field,
 the application of the low-rank approximation method utilizing symmetry in the proxy method is straightforward.
 Furthermore, we can construct an efficient fast direct solver using symmetry, 
 in the same manner as its Helmholtz counterpart.
For the details of the proxy method with Galerkin discretization for elastic transmission problems,
 see \cite{matsumoto2026fast}.

\begin{remark}
Although near-field singularities require careful regularization to ensure overall matrix symmetry \cite{MARUYAMA202511-251219},
 the proposed low-rank approximation remains effective without issue.
 This is because the method exploits symmetry within the interaction between a partial boundary and its surrounding proxy boundary, which corresponds to a far-field, during the proxy method.
\end{remark}

\section{Numerical experiments} \label{sec:numerical_demo}
We demonstrate the performance of the proposed fast direct solver based on weak admissibility exploiting symmetry.
In Helmholtz scattering and elastic scattering, a plane wave and a longitudinal plane wave propagating in the $x_1$ direction are respectively used as incident waves.
For the numerical experiments, the material parameters for the Helmholtz scattering problem are set to $\varepsilon_0 = 1$ and $\varepsilon_1 = 3$.
 For the elastic scattering problem, the parameters are chosen as $\rho^{(0)} = 1$, $\lambda^{(0)} = 1$, $\mu^{(0)} = 1$, $\rho^{(1)} = 2$, $\lambda^{(1)} = 9$, and $\mu^{(1)} = 4.5$.

As the computational resource,
multiple compute nodes of the supercomputer TSUBAME4.0 at the Institute of Science Tokyo,
each with two 96-core AMD EPYC 9654 CPUs (192 cores total;
maximum frequency: 3708 MHz) and 768 GB of 0.92-TB/s DDR5 memory,
were used.
The proposed method was implemented using a hybrid MPI/OpenMP parallelization scheme \cite{matsumoto2026proxy}, by which each MPI process is associated with a single compute node.
 Furthermore, each process manages a specific subtree within the global tree structure, and its internal tasks are parallelized via OpenMP.
The code was implemented in C++, except for the calculating the discretized layer potential for elastic scattering, which was written in Fortran.

\subsection{Helmholtz transmission problems}

\begin{example}[Verification] \label{ex:hel_error}
We first verify the correctness of the implementation of the proposed fast direct solver exploiting symmetry.
Two compute nodes, corresponding to two MPI processes (with a total of 384 CPU cores), are used for this verification.
Following Theorem \ref{thm:order}, we set the initial rank to $k_L = 30$ and the rank scaling parameter $a = 1.15$. 
The standard M\"uller-like boundary integral equation, which corresponds to a non-scaled version of \eqref{eq:symmetric_muller}, is given by
\begin{equation}
  \mqty[
  -\qty(\varepsilon_{0}^{2} \bm{S}_{k_{0}}^\prime - \varepsilon_{1}^2 \bm{S}_{k_1}^\prime)
  &
  -\frac{\varepsilon_{0} + \varepsilon_{1}}{2}I + \varepsilon_{0} \bm{D}_{k_0}^\prime - \varepsilon_{1} \bm{D}_{k_1}^\prime
  \\
  \frac{\varepsilon_{0} + \varepsilon_{1}}{2}I + \varepsilon_{0} \bm{D}_{k_0}^{\ast \prime}
  - \varepsilon_{1} \bm{D}_{k_1}^{\ast \prime}
  &
  -\qty(\bm{N}_{k_0}^\prime - \bm{N}_{k_1}^\prime)
  ]
  \mqty[
  \bm{q} \\
  \bm{u}
  ]
  =
  \mqty[
  -\varepsilon_0 \bm{u}^{\mathrm{in}} \\
  \frac{1}{\varepsilon_{0}} \bm{q}^{\mathrm{in}}
  ], \label{eq:naive_muller}
\end{equation}
where $I$ is the identity matrix.
We use \eqref{eq:naive_pmchwt} and \eqref{eq:naive_muller} as reference formulations compared with the symmetric counterparts, \eqref{eq:symmetric_pmchwt} and \eqref{eq:symmetric_muller}.

Using the analytical solution for a unit circle,
the relative 2-norm errors of the numerical solution to it are calculated at the quadrature points, shown in Figure \ref{fig:helm_error}.
In this and subsequent figures, the degrees of freedom of the system are always twice the number of quadrature points.
As shown in Figure \ref{fig:helm_error}, the M\"uller-like and symmetric M\"uller-like (M\"uller-like Sym) formulations achieve $O(h^3)$ convergence, whereas the PMCHWT and symmetric PMCHWT (PMCHWT Sym) formulations exhibit only $O(h)$ convergence, where $h$ denotes the interval between quadrature points on the parameterized boundary.
Moreover, we see that the symmetric formulations do not degrade accuracy.
Furthermore, when the ratio of the frequency to the degrees of freedom is kept constant,
 the computational accuracy is also preserved.
These findings demonstrate the correctness of our implementation of the fast direct solver.
\begin{figure}[tb]
  \centering
  \includegraphics[width = 0.48\linewidth]{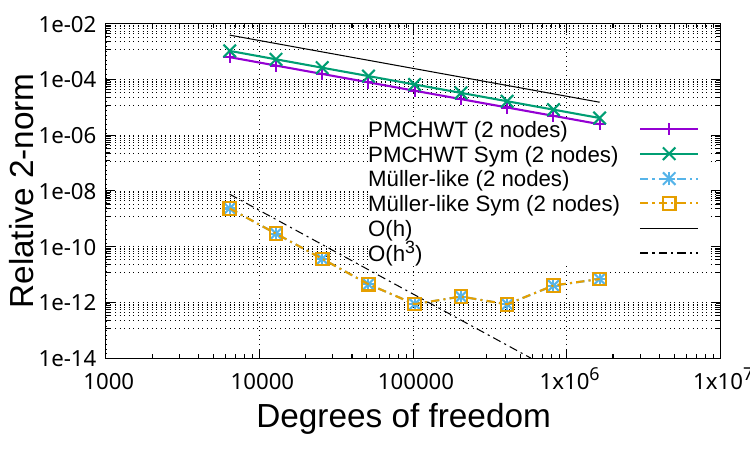}
  \hfill
  \includegraphics[width = 0.48\linewidth]{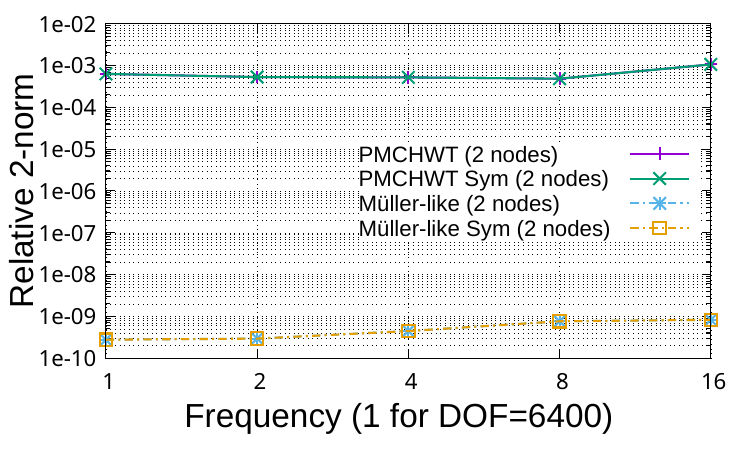}
  \caption{Relative 2-norm errors for unit circle in Helmholtz scattering. Left: errors obtained by varying the problem size while maintaining the frequency at $\omega = 2$. Right: errors obtained by maintaining a constant ratio between frequency and degrees of freedom (DOF).}
  \label{fig:helm_error}
\end{figure}

\end{example}

\begin{example}[Speedup]
We next measure the computational efficiency of the solver using the kite geometry shown in Figure \ref{fig:geometries}.
We again use an initial rank of 30 and the rank scaling parameter $a = 1.15$ as in Example \ref{ex:hel_error}.
Specifically, we compare the computational speeds of the symmetric \eqref{eq:symmetric_muller} and standard \eqref{eq:naive_muller} Müller-like formulations against Burton--Miller formulation \eqref{eq:bm} for the Helmholtz transmission problem.
The Burton--Miller formulation is given by
\begin{equation}
  \mqty[
    \bm{D}_{k_0}^\prime -\frac{I}{2} + \alpha \bm{N}_{k_0}^\prime
  &
    -\varepsilon_0 \bm{S}_{k_0}^\prime  + \alpha \varepsilon_0 \qty(\bm{D}_{k_0}^{* \prime} + \frac{I}{2})
  \\
  \bm{D}_{k_1}^\prime + \frac{I}{2}
  &
  -\varepsilon_1 \bm{S}_{k_1}^\prime
  ]
  \mqty[
  \bm{u} \\
  \bm{q}
  ]
  =
  \mqty[
  -\bm{u}^{\mathrm{in}} -\alpha \bm{q}^{\mathrm{in}} \\
  0
  ], \label{eq:bm}
\end{equation}
where $\alpha = i/k_0$ is a constant of the Burton--Miller methods.
This formulation requires only six layer potentials, which is fewer than the eight required by the PMCHWT and original Müller formulations.
 In fact, for the Burton--Miller formulation, it has been reported that the numerical computation can be accelerated by approximately 20\% when using a fast direct solver \cite{matsumoto2026fast}.

Figure \ref{fig:speed_helm} presents the computational times for each formulation. Notably, the symmetric M\"uller-like formulation achieves the shortest computational time. This demonstrates that, despite maintaining the same asymptotic order of floating-point operations, the symmetric formulation successfully reduces the leading constant factors, thereby fulfilling the intended objective. Although the Burton--Miller formulation outperforms the standard M\"uller-like formulation, it remains less efficient than its symmetric M\"uller-like counterpart.
When using eight compute nodes, the performance relationship among the formulations is less distinct than with fewer nodes.
This is likely since the computational resources are excessive relative to the workload.
 Even in this case, the symmetric M\"uller-like formulation is generally the fastest.
\begin{figure}[tb]
  \centering
  \includegraphics[width=0.48\textwidth]{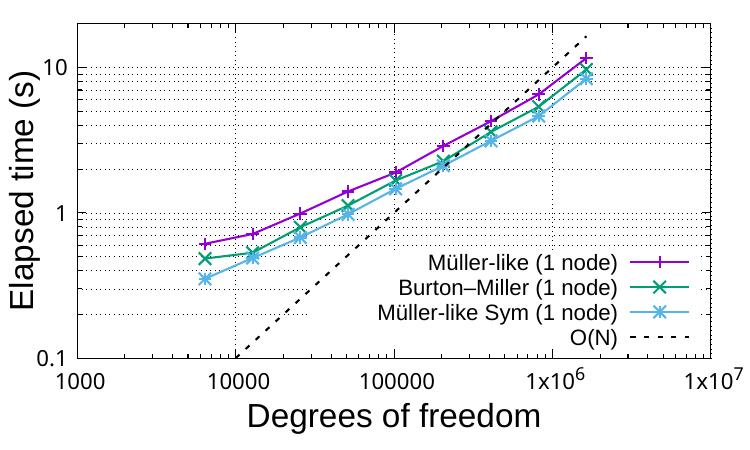}
  \hfill
  \includegraphics[width=0.48\textwidth]{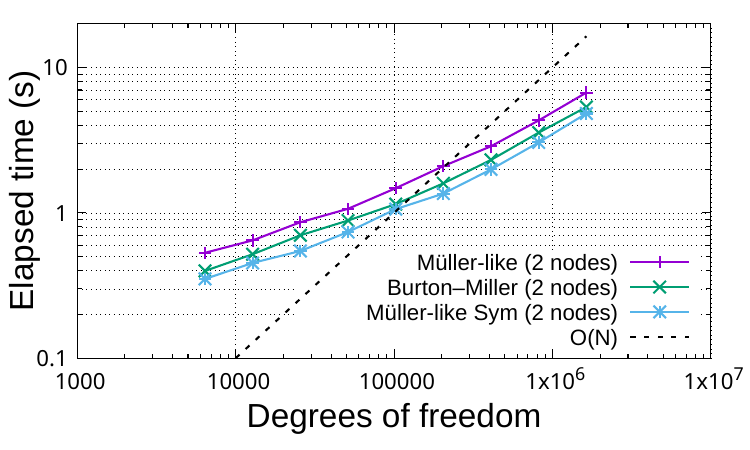}
  \\
  \includegraphics[width=0.48\textwidth]{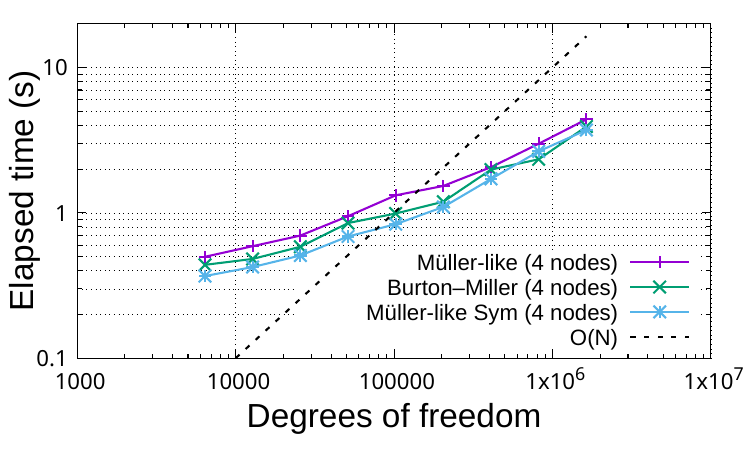}
  \hfill
  \includegraphics[width=0.48\textwidth]{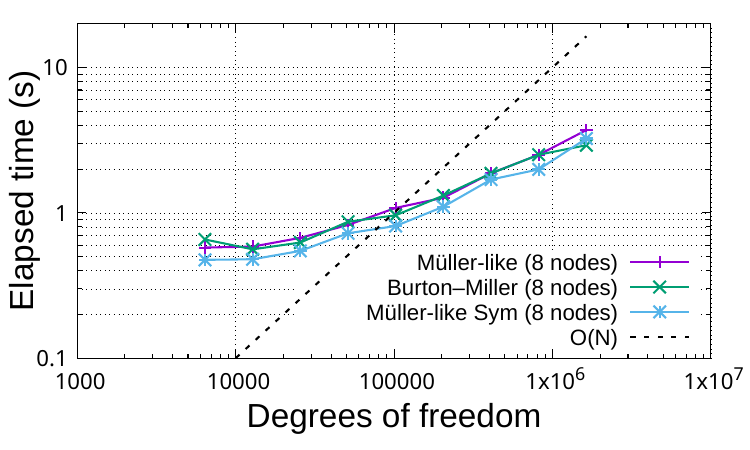}
  \caption{Computational time obtained using kite geometry in Helmholtz scattering.
    The four panels show the elapsed times of each boundary integral formulation when the number of compute nodes is varied as 1, 2, 4, and 8 for the same problem setting.
  }
  \label{fig:speed_helm}
\end{figure}
The speedup ratios of the symmetric M\"uller-like and Burton--Miller formulations relative to the standard M\"uller-like baseline are summarized in Table \ref{tab:speedup_ratio_helmholtz}; these are the same results for 1, 2, and 4 nodes presented  in Figure \ref{fig:speed_helm}.
 This table indicates that the symmetric M\"uller-like formulation achieves a speedup ratio of approximately 1.41 in the Helmholtz transmission problem, demonstrating higher efficiency than that of approximately 1.23 in the Burton--Miller formulation.
\begin{table}[htbp]
  \centering
  \caption{Speedup ratios of Burton--Miller (BM) and symmetric M\"uller-like (M-sym) formulations compared with standard M\"uller-like formulation (Base) in Helmholtz transmission wave scattering. We define the speedup ratio of ``Base'' as 1.0.}
  \label{tab:speedup_ratio_helmholtz}
  \resizebox{0.8\linewidth}{!}{
  \begin{tabular}{l ccc ccc ccc}
    \toprule
    Number of nodes & \multicolumn{3}{c}{1} & \multicolumn{3}{c}{2} & \multicolumn{3}{c}{4} \\
    (CPU cores) & \multicolumn{3}{c}{(192)} & \multicolumn{3}{c}{(384)} & \multicolumn{3}{c}{(768)} \\
    \cmidrule(lr){2-4} \cmidrule(lr){5-7} \cmidrule(lr){8-10}
    Degrees of freedom & Base & BM & M-sym & Base & BM & M-sym & Base & BM & M-sym \\
    \midrule
    6400     & 1.0 & 1.264 & 1.741 & 1.0 & 1.333 & 1.513 & 1.0 & 1.136 & 1.357 \\ 
    12800    & 1.0 & 1.346 & 1.464 & 1.0 & 1.246 & 1.431 & 1.0 & 1.224 & 1.392 \\ 
    25600    & 1.0 & 1.242 & 1.461 & 1.0 & 1.230 & 1.577 & 1.0 & 1.194 & 1.371 \\ 
    51200    & 1.0 & 1.251 & 1.435 & 1.0 & 1.204 & 1.444 & 1.0 & 1.115 & 1.379 \\ 
    102400   & 1.0 & 1.135 & 1.294 & 1.0 & 1.287 & 1.393 & 1.0 & 1.334 & 1.580 \\ 
    204800   & 1.0 & 1.274 & 1.370 & 1.0 & 1.314 & 1.545 & 1.0 & 1.283 & 1.395 \\ 
    409600   & 1.0 & 1.185 & 1.365 & 1.0 & 1.241 & 1.443 & 1.0 & 1.040 & 1.205 \\ 
    819200   & 1.0 & 1.215 & 1.410 & 1.0 & 1.213 & 1.418 & 1.0 & 1.279 & 1.129 \\ 
    1638400  & 1.0 & 1.199 & 1.386 & 1.0 & 1.257 & 1.393 & 1.0 & 1.121 & 1.184 \\ 
    \midrule
    Average  & 1.0 & 1.234 & 1.436 & 1.0 & 1.258 & 1.462 & 1.0 & 1.192 & 1.332 \\
    \bottomrule
  \end{tabular}
  }
\end{table}

\end{example}

\subsection{Elastic transmission problem}
\begin{example}[Verification]
We analogously test the solver based on the symmetric formulation for the elastic transmission problem.
 For this problem, we compare the symmetric formulation against a straightforward low-rank approximation of PMCHWT equations \eqref{eq:elastic_pmchwt} that does not exploit symmetry. This baseline corresponds to the case where the double-sided interpolative decomposition is naively performed. The legends ``PMCHWT Elastic Sym'' and ``PMCHWT Elastic'' in the figures indicate the solvers based on the efficient and the naive low-rank approximations, respectively.

\begin{figure}[tb]
  \centering
  \includegraphics[width = 0.48\linewidth]{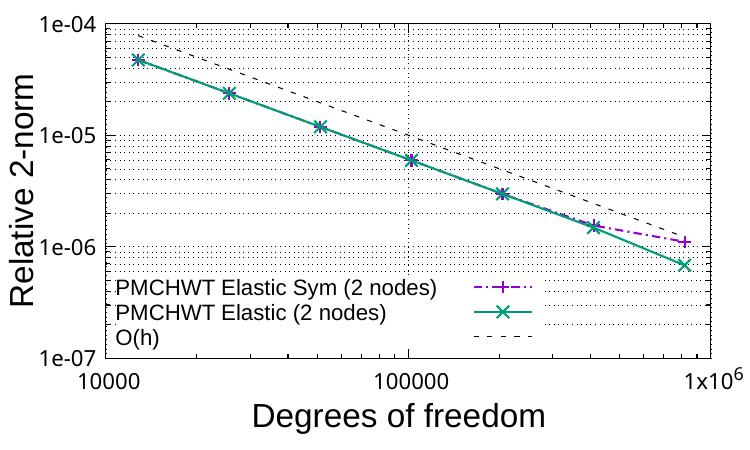}
  \hfill
  \includegraphics[width = 0.48\linewidth]{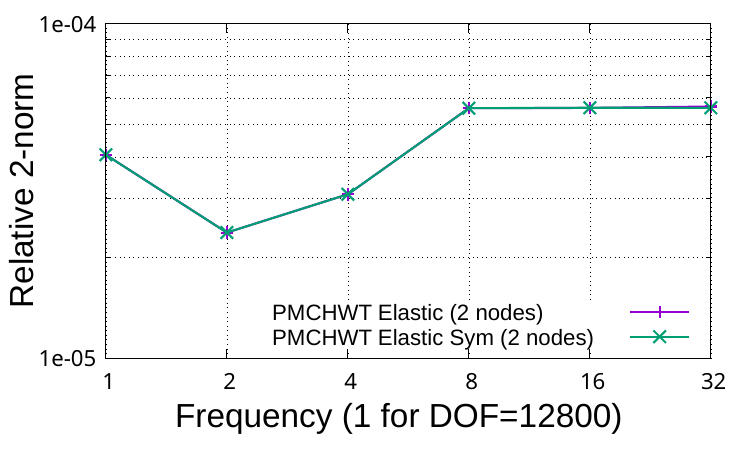}
  \label{fig:elastic_error}
  \caption{Relative 2-norm errors of direct solver for  unit circle in each formulation at initial rank $30$ and rank scaling parameter $a = 1.15$.
    Left: errors obtained by varying the problem size while maintaining the frequency at $\omega = 2$. Right: errors obtained by maintaining a constant ratio between frequency and degrees of freedom (DOF).}
\end{figure}
We first use the initial rank $30$ and the rank scaling parameter $a = 1.15$, which is the same to the Helmholtz counterpart.
The relative 2-norm errors of the direct solver relative to the analytical solution are shown in Figure \ref{fig:elastic_error}, which includes the results at a fixed frequency and those where the frequency scales with the degrees of freedom.
As shown, implementation of the low-rank approximation based on the proxy method and its application to the fast direct solver have been mostly verified.
However, at 819200 degrees of freedom in the left of Figure \ref{fig:elastic_error}, 
the convergence rate deviates from $O(h)$, where $h$ is the representative mesh size.

\begin{figure}[tb]
  \centering
  \includegraphics[width=0.48\textwidth]{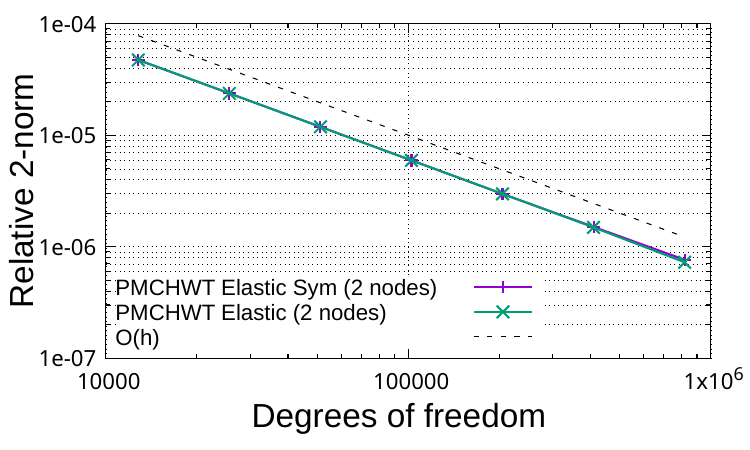}
  \hfill
  \includegraphics[width=0.48\textwidth]{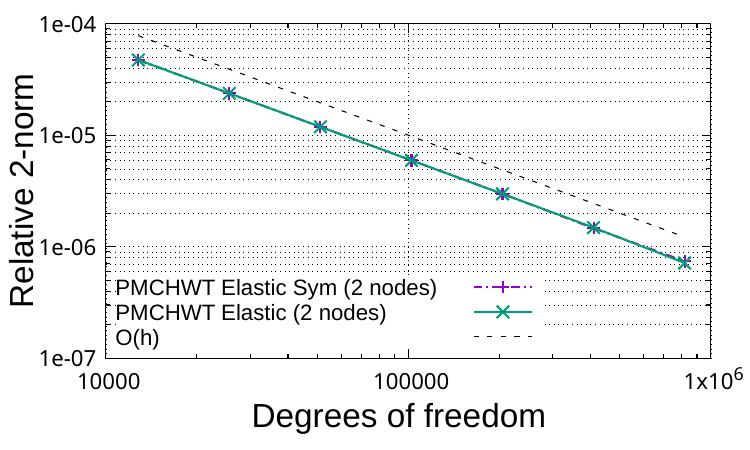}
  \includegraphics[width=0.48\textwidth]{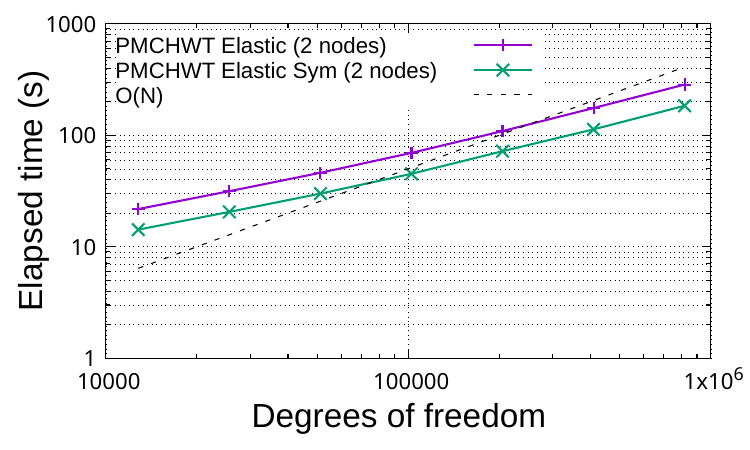}
  \hfill
  \includegraphics[width=0.48\textwidth]{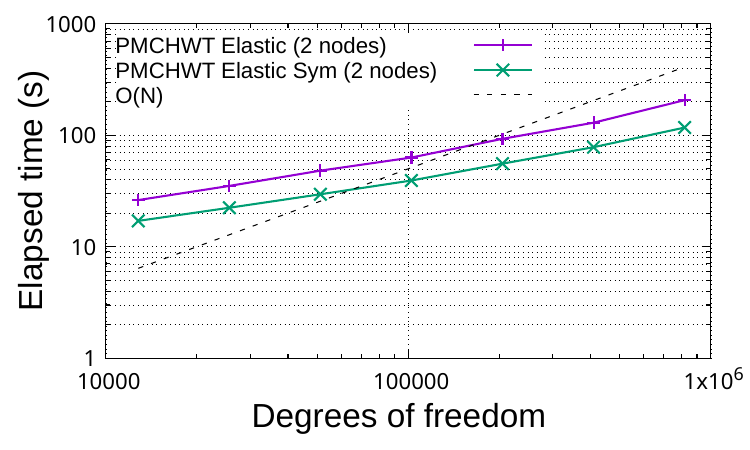}
  \caption{
    Relative 2-norm errors (top) and computational time (bottom) for unit circle in elastic wave scattering.
    Left column: initial rank $30$, scaling parameter $a = 1.26$.
    Right column: initial rank $40$, scaling parameter $a = 1.15$.
  }
  \label{fig:elastic_verify_other_prams}
\end{figure}
We then vary the parameters of the low-rank approximation to enhance the accuracy.
 Figure \ref{fig:elastic_verify_other_prams} indicates that increasing the initial rank is preferable. Although a larger scaling parameter also improves accuracy, it significantly increases computational time.

\end{example}

\begin{example}[Speedup]
Using $\omega = 2$, initial rank $40$ and the rank scaling parameter $a = 1.15$,
we measure the performance of the fast direct solver exploiting symmetry in elastic transmission problems.
As the inclusion, the square geometry in Figure \ref{fig:geometries} is used.
Similar to the Helmholtz scattering case, the Burton--Miller formulation for the elastic transmission problem is used, which is given by
\begin{equation}
  \mqty[
    \bm{T}^{(0)} - \frac{M_{11}}{2} + \beta \bm{H}^{(0)}
  &
    -\bm{U}^{(0)}  + \beta \qty(\bm{T}^{* (0)} - \frac{M_{10}}{2})
  \\
  \bm{T}^{(1)} + \frac{M_{01}}{2}
  &
  -\bm{U}^{(1)}
  ]
  \mqty[
  \bm{v} \\
  \bm{t}
  ]
  =
  \mqty[
  -\bm{v}^{\mathrm{in}} -\beta \bm{t}^{\mathrm{in}} \\
  0
  ], \label{eq:elastic_bm}
\end{equation}
where $M_{11}$, $M_{10}$, and $M_{01}$ denote the Gram matrices associated with the P1 test and P1 basis functions, the P1 test and P0 basis functions, and the P0 test and P1 basis functions, respectively.
Here, $\beta = i / ( \omega \sqrt{\mu^{(0)} \rho^{(0)}})$ is a Burton--Miller constant.

\begin{figure}[tb]
  \centering
  \includegraphics[width=0.48\textwidth]{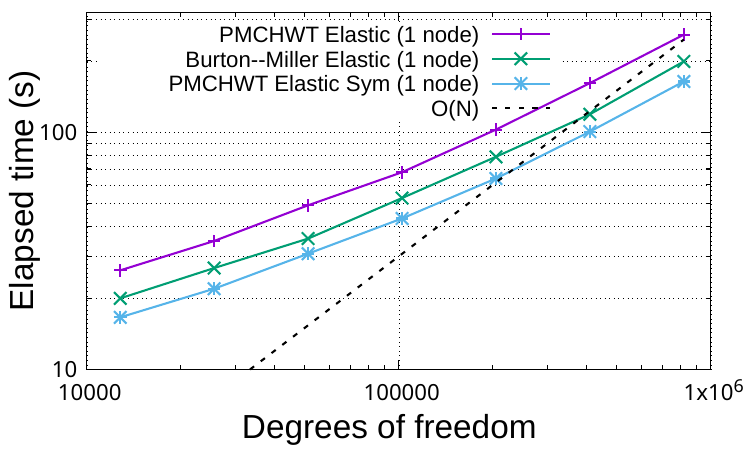}
  \hfill
  \includegraphics[width=0.48\textwidth]{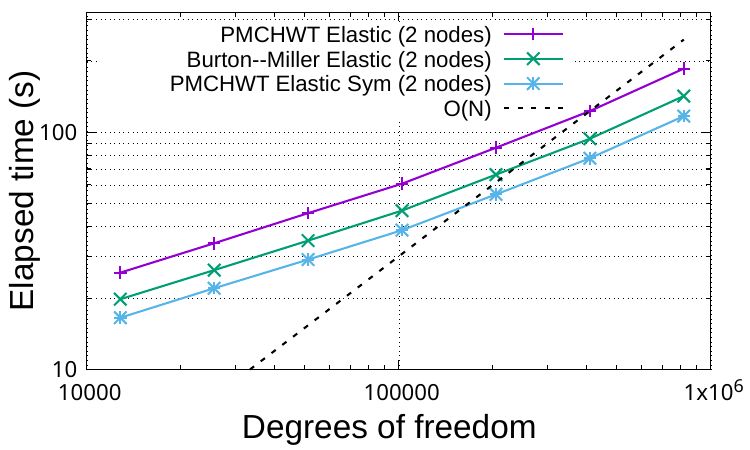}
  \\
  \includegraphics[width=0.48\textwidth]{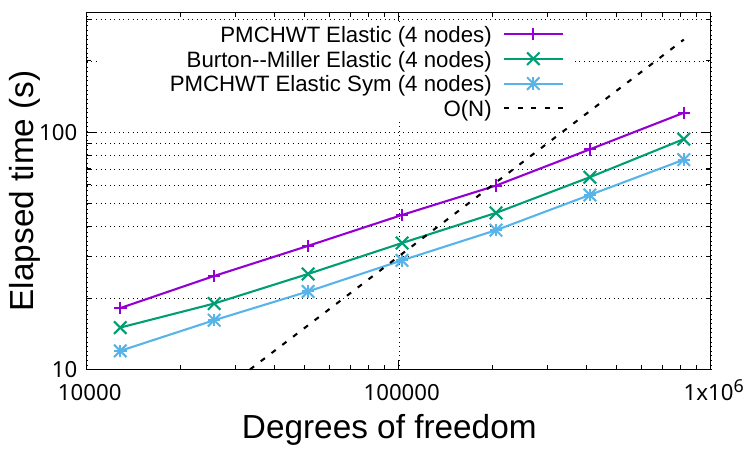}
  \hfill
  \includegraphics[width=0.48\textwidth]{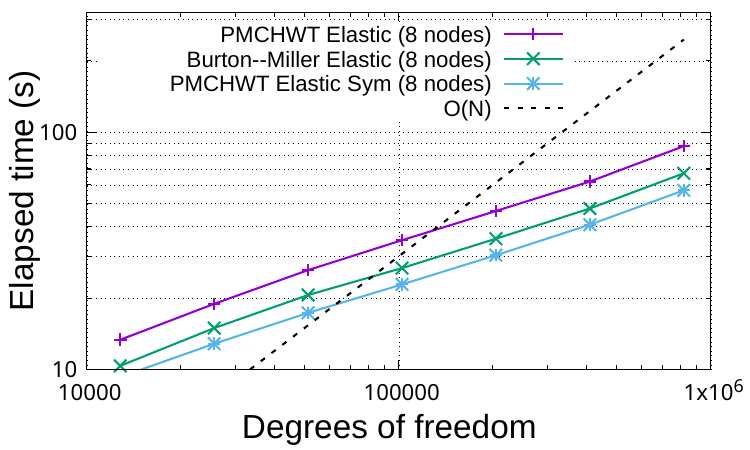}
\caption{
  Computational time obtained using square geometry in elastic wave scattering.
  The four plots show the elapsed times of
  the direct solvers based on the PMCHWT \eqref{eq:elastic_pmchwt} and Burton--Miller \eqref{eq:elastic_bm} formulations when varying the number of compute nodes as 1, 2, 4, and 8 under the same problem setting.
}
  \label{fig:time_elastic}
\end{figure}
Figure \ref{fig:time_elastic} shows the results with respect to computational time.
 Due to the complex nature of the fundamental solution, the solver for elastic wave scattering requires a higher computational workload than its Helmholtz counterpart.
 This ensures that the PMCHWT formulation exploiting symmetry remains the fastest, even when using eight compute nodes.
 Table \ref{tab:speedup_ratio_elastic} presents the speedup ratios corresponding to Figure \ref{fig:time_elastic}.
 As indicated in this table, the fast direct solver based on the PMCHWT formulation exploiting symmetry is over 1.5 times faster than the one not using symmetry, and the Burton--Miller formulation achieves approximately a 1.3-fold speedup.

\begin{table}[htbp]
  \centering
  \caption{Speedup ratios of Burton--Miller (BM) and symmetric PMCHWT (P-sym) formulations relative to standard PMCHWT formulation (Base) in elastic wave transmission scattering. We define the speedup ratio of ``Base'' as 1.0.}
  \label{tab:speedup_ratio_elastic}
  \resizebox{\linewidth}{!}{
  \begin{tabular}{l ccc ccc ccc ccc}
    \toprule
    Number of nodes & \multicolumn{3}{c}{1} & \multicolumn{3}{c}{2} & \multicolumn{3}{c}{4} & \multicolumn{3}{c}{8} \\
    (CPU cores) & \multicolumn{3}{c}{(192)} & \multicolumn{3}{c}{(384)} & \multicolumn{3}{c}{(768)} & \multicolumn{3}{c}{(1536)} \\
    \cmidrule(lr){2-4} \cmidrule(lr){5-7} \cmidrule(lr){8-10} \cmidrule(lr){11-13} 
    Degrees of freedom & Base & BM & P-sym & Base & BM & P-sym & Base & BM & P-sym & Base & BM & P-sym \\
    \midrule
    12800  & 1.0 & 1.312 & 1.577 & 1.0 & 1.294 & 1.549 & 1.0 & 1.209 & 1.515 & 1.0 & 1.288 & 1.435 \\
    25600  & 1.0 & 1.301 & 1.590 & 1.0 & 1.298 & 1.548 & 1.0 & 1.307 & 1.536 & 1.0 & 1.265 & 1.473 \\
    51200  & 1.0 & 1.382 & 1.600 & 1.0 & 1.306 & 1.570 & 1.0 & 1.314 & 1.557 & 1.0 & 1.273 & 1.514 \\
    102400 & 1.0 & 1.284 & 1.566 & 1.0 & 1.298 & 1.570 & 1.0 & 1.313 & 1.553 & 1.0 & 1.311 & 1.536 \\
    204800 & 1.0 & 1.302 & 1.608 & 1.0 & 1.298 & 1.572 & 1.0 & 1.303 & 1.541 & 1.0 & 1.308 & 1.536 \\
    409600 & 1.0 & 1.353 & 1.601 & 1.0 & 1.308 & 1.581 & 1.0 & 1.310 & 1.556 & 1.0 & 1.299 & 1.520 \\
    819200 & 1.0 & 1.297 & 1.575 & 1.0 & 1.306 & 1.586 & 1.0 & 1.291 & 1.577 & 1.0 & 1.301 & 1.535 \\
    \midrule
    Average & 1.0 & 1.319 & 1.588 & 1.0 & 1.301 & 1.568 & 1.0 & 1.292 & 1.548 & 1.0 & 1.292 & 1.507 \\
    \bottomrule
  \end{tabular}
  }
\end{table}

\end{example}

\section{Conclusions} \label{sec:conclusions}
In this paper, we presented an efficient framework for accelerating hierarchical low-rank compression in fast direct solvers by fully exploiting matrix symmetry in wave transmission scattering problems.
For Helmholtz scattering, we proposed a symmetric Müller-like boundary integral equation that involves at most weakly singular kernels and maintains a minimal number of unknowns while avoiding fictitious eigenvalues.
Combined with a one-point zeta-corrected Nyström discretization, the formulation successfully achieves third-order convergence while preserving complex symmetry.
Within the context of the HSS representation-based direct solver, this symmetry allows us to halve the computational cost of the interpolative decomposition in the proxy method by setting the left coefficient matrix for low-rank approximation as the transpose of the right one.
Furthermore, this paper has provided guidance on how to select scaling parameters in multilevel solvers in Theorem \ref{thm:order},
which have traditionally been determined empirically.
We have also demonstrated the extensibility of this framework to two-dimensional elastic wave transmission problems via the symmetric Galerkin boundary element method applied to the PMCHWT formulation.
Comprehensive numerical experiments on a supercomputer validated the correctness and efficiency of our implementation,
showing significant speedups, approximately 1.41-fold for Helmholtz scattering and greater than 1.5-fold for elastic scattering, without loss of accuracy.

 Future avenues of research include expanding this approach to three-dimensional problems and extending the symmetry-exploiting techniques to more advanced hierarchical solvers, such as $\mathcal{H}^2$-matrix-based direct solvers and hierarchical interpolative factorizations incorporating an additional re-compression phase. 
Furthermore, although the present study only treated the single-scatterer case, the proposed framework can be extended to cases involving multiple scatterers with straightforward modifications following \cite{matsumoto2026accelerated}.

\bibliographystyle{siamplain}
\bibliography{siamref_short}
\end{document}

%% file: shared_arxiv.tex
\usepackage{lipsum}
\usepackage{amsfonts}
\usepackage{graphicx}
\usepackage{epstopdf}
\usepackage{algorithmic}
\ifpdf
  \DeclareGraphicsExtensions{.eps,.pdf,.png,.jpg}
\else
  \DeclareGraphicsExtensions{.eps}
\fi

\newsiamremark{remark}{Remark}
\newsiamremark{hypothesis}{Hypothesis}
\crefname{hypothesis}{Hypothesis}{Hypotheses}
\newsiamthm{claim}{Claim}
\newsiamremark{fact}{Fact}
\crefname{fact}{Fact}{Facts}

\headers{Hierarchical compression exploiting symmetry}{Y. Matsumoto, T. Maruyama, Q. Ma and R. Yokota}

\title{Hierarchical low-rank compression exploiting symmetric boundary integral equations in transmission scattering\thanks{Submitted to the editors DATE.
\funding{This work was supported by Japan Society for the Promotion of Science under KAKENHI grant number 24K20783
and by the computational resources of TSUBAME4.0 at the Institute of Science Tokyo %and Camphor3 from Kyoto University,
provided through the projects ``Joint Usage/Research Center for Interdisciplinary Large-scale Information Infrastructures (JHPCN)'' and ``High Performance Computing Infrastructure (HPCI)'' in Japan (project IDs jh260048 and jh260068, respectively).}}}

\author{
  Yasuhiro Matsumoto\thanks{Center for Information Infrastructure, Institute of Science Tokyo, Tokyo, Japan
    (\email{matsumoto@cii.isct.ac.jp}).}
  \and Taizo Maruyama\thanks{Department of Civil and Environmental Engineering, Institute of Science Tokyo, Tokyo, Japan
    (\email{maruyama.t.45ef@m.isct.ac.jp}).}
  \and Qianxiang Ma\thanks{RIKEN Center for Computational Science, Kobe, Japan
    (\email{qianxiang.ma@riken.jp}).}
\and Rio Yokota\thanks{Institute of Integrated Research, Institute of Science Tokyo, Tokyo, Japan
    (\email{rioyokota@rio.scrc.iir.isct.ac.jp}).}
}

\usepackage{amsopn}
\DeclareMathOperator{\diag}{diag}

%% file: siamref_short.bib
@article{MATSUMOTO2025106148,
title = {Linearly scalable fast direct solver based on proxy surface method for two-dimensional elastic wave scattering by cavity},
journal = {Eng. Anal. Bound. Elem.},
volume = 173,
pages = 106148,
year = 2025,
issn = {0955-7997},
doi = {10.1016/j.enganabound.2025.106148},
author = {Yasuhiro Matsumoto and Taizo Maruyama},
Xurl = {https://www.sciencedirect.com/science/article/pii/S0955799725000360},
}

@article{matsumoto2026fast,
  title={A fast direct solver for two-dimensional transmission problems of elastic waves},
  author={Matsumoto, Yasuhiro and Maruyama, Taizo},
  journal={Eng. Comput.},
  volume={42},
  number={68},
  pages={},
  year={2026},
doi={10.1007/s00366-026-02302-8},
  publisher={Springer}
}

@article{matsumoto2026proxy,
  title={Proxy-surface-based fast direct solver for {TE}-mode scattering problems on distributed memory systems},
  author={Matsumoto, Yasuhiro and Yokota, Rio},
  journal={arXiv preprint arXiv:2607.00790},
doi={10.48550/arXiv.2607.00790},
  year={2026}
}

@article{matsumoto2025efficient,
  title={Efficient {LU} factorization exploiting direct-indirect {B}urton--{M}iller equation for {H}elmholtz transmission problems},
  author={Matsumoto, Yasuhiro and Matsushima, Kei},
  journal={arXiv preprint arXiv:2512.14193},
doi={10.48550/arXiv.2512.14193},
  year={2025}
}

@book{colton2013inverse,
  title={Inverse acoustic and electromagnetic scattering theory (Third Edition)},
  author={Colton, David L and Kress, Rainer},
  year={2013},
  address={New York},
  doi={10.1007/978-1-4614-4942-3},
  publisher={Springer}
}

@book{kress2014linear,
  title={Linear Integral Equations (Third Edition)},
  author={Kress, Rainer},
  year={2014},
  address={New York},
doi={10.1007/978-1-4614-9593-2},
  publisher={Springer}
}

@article{kress1977transmission,
    author = {Kress, R. and Roach, G. F.},
    title = {Transmission problems for the {H}elmholtz equation},
    journal = {J. Math. Phys.},
    volume = {19},
    number = {6},
    pages = {1433-1437},
    year = {1978},
    month = {06},
    issn = {0022-2488},
    doi = {10.1063/1.523808},
}

@article{MARUYAMA202511-251219,
  title={On Numerical Symmetry of Coefficient Matrix in the {PMCHWT} Formulation for Elastodynamic Transmission Problem},
  author={Taizo Maruyama and Yasuhiro Matsumoto},
  journal={Trans. JASCOME},
  volume={25},
  number={ },
  pages={87-93},
  year={2025},
  doi={10.60443/jascome.25.0_87},
note={(in Japanese)}
}

@book{sutradhar2008symmetric,
  title={Symmetric Galerkin boundary element method},
  author={Sutradhar, Alok and Paulino, Glaucio H and Gray, Leonard J},
  year={2008},
  publisher={Springer Berlin},
  address={Heidelberg},
doi={10.1007/978-3-540-68772-6},
}

@article{hiptmair2022spurious,
  title={Spurious quasi-resonances in boundary integral equations for the {H}elmholtz transmission problem},
  author={Hiptmair, Ralf and Moiola, Andrea and Spence, Euan A},
  journal={SIAM J. Appl. Math.},
  volume={82},
  number={4},
  pages={1446--1469},
  year={2022},
doi={10.1137/21M1447052},
  publisher={SIAM}
}

@incollection{POGGIO1973159,
title = {CHAPTER 4 - {I}ntegral Equation Solutions of Three-dimensional Scattering Problems},
editor = {R. Mittra},
booktitle = {Computer Techniques for Electromagnetics},
publisher = {Pergamon},
  address		= "Oxford",
pages = {159-264},
year = {1973},
series = {International Series of Monographs in Electrical Engineering},
isbn = {978-0-08-016888-3},
doi = {10.1016/B978-0-08-016888-3.50008-8},
author = {A.J. Poggio and E.K. Miller}
}

@ARTICLE{Chang1977surface,
  author={Chang, Y. and Harrington, R.},
  journal={IEEE Trans. Antennas Propag.}, 
  title={A surface formulation for characteristic modes of material bodies}, 
  year={1977},
  volume={25},
  number={6},
  pages={789-795},
  doi={10.1109/TAP.1977.1141685}
}

@article{VANTWOUT2022111229,
title = {Frequency-robust preconditioning of boundary integral equations for acoustic transmission},
journal = {J. Comput. Phys.},
volume = {462},
pages = {111229},
year = {2022},
issn = {0021-9991},
doi = {10.1016/j.jcp.2022.111229},
author = {Elwin {van 't Wout} and Seyyed R. Haqshenas and Pierre Gélat and Timo Betcke and Nader Saffari},
}

@article{Antoine01102008,
author = {X. Antoine and Y. Boubendir},
title = {An integral preconditioner for solving the two-dimensional scattering transmission problem using integral equations},
journal = {Int. J. Comput. Math.},
volume = {85},
number = {10},
pages = {1473--1490},
year = {2008},
publisher = {Taylor \& Francis},
doi = {10.1080/00207160802033335},
}

@article{NIINO201266,
title = {Preconditioning based on {C}alderon’s formulae for periodic fast multipole methods for {H}elmholtz’ equation},
journal = {J. Comput. Phys.},
volume = {231},
number = {1},
pages = {66-81},
year = {2012},
issn = {0021-9991},
doi = {10.1016/j.jcp.2011.08.019},
author = {Kazuki Niino and Naoshi Nishimura},
}

@inproceedings{ma2022scalable,
  author={Ma, Qianxiang and Deshmukh, Sameer and Yokota, Rio},
  booktitle={SC22: International Conference for High Performance Computing, Networking, Storage and Analysis}, 
  title={Scalable Linear Time Dense Direct Solver for 3-{D} Problems without Trailing Sub-Matrix Dependencies}, 
  year={2022},
  volume={},
  number={},
  pages={1-12},
  doi={10.1109/SC41404.2022.00088}
}

@article{ma2024inherently,
  title={An inherently parallel $\mathcal{H}$2-{ULV} factorization for solving dense linear systems on {GPU}s},
  author={Ma, Qianxiang and Yokota, Rio},
  journal={Int. J. High Perform. Comput. Appl.},
  volume={38},
  number={4},
  pages={314--336},
  year={2024},
doi={10.1177/10943420241242021},
  publisher={SAGE Publications Sage UK: London, England}
}

@article{chandrasekaran2005Calcolo,
  title={A fast adaptive solver for hierarchically semiseparable representations},
  author={Chandrasekaran, Shivkumar and Gu, Ming and Lyons, William},
  journal={Calcolo},
  volume={42},
  number={3},
  pages={171--185},
  year={2005},
doi={10.1007/s10092-005-0103-3},
  publisher={Springer}
}

@article{chandrasekaran2006fast,
  title={A fast {ULV} decomposition solver for hierarchically semiseparable representations},
  author={Chandrasekaran, Shiv and Gu, Ming and Pals, Timothy},
  journal={SIAM J. Matrix Anal. Appl.},
  volume={28},
  number={3},
  pages={603--622},
  year={2006},
doi={10.1137/S0895479803436652},
  publisher={SIAM}
}

@article{wu2021corrected,
  title={Corrected trapezoidal rules for boundary integral equations in three dimensions},
  author={Wu, Bowei and Martinsson, Per-Gunnar},
  journal={Numer. Math.},
  volume={149},
  number={4},
  pages={1025--1071},
  year={2021},
doi={10.1007/s00211-021-01244-1},
  publisher={Springer}
}

@article{wu2021zeta,
  title={Zeta correction: a new approach to constructing corrected trapezoidal quadrature rules for singular integral operators},
  author={Wu, Bowei and Martinsson, Per-Gunnar},
  journal={Adv. Comput. Math.},
  volume={47},
  number={3},
  pages={45},
  year={2021},
  doi={10.1007/s10444-021-09872-9},
  publisher={Springer}
}

@article{wu2023unified,
  title={A unified trapezoidal quadrature method for singular and hypersingular boundary integral operators on curved surfaces},
  author={Wu, Bowei and Martinsson, Per-Gunnar},
  journal={SIAM J. Numer. Anal.},
  volume={61},
  number={5},
  pages={2182--2208},
  year={2023},
  doi={10.1137/22M1520372},
  publisher={SIAM}
}

@book{martinsson2019book,
author = {Martinsson, Per-Gunnar},
title = {Fast Direct Solvers for Elliptic PDEs},
publisher = {Society for Industrial and Applied Mathematics},
year = {2019},
doi = {10.1137/1.9781611976045},
address = {Philadelphia, PA},
edition   = {},
}

@article{MARTINSSON20051,
title = {A fast direct solver for boundary integral equations in two dimensions},
journal = {J. Comput. Phys.},
volume = {205},
number = {1},
pages = {1-23},
year = {2005},
issn = {0021-9991},
doi = {10.1016/j.jcp.2004.10.033},
Xurl = {https://www.sciencedirect.com/science/article/pii/S0021999104004462},
author = {Per-Gunnar Martinsson and V. Rokhlin},
}

@article{martinsson2007fast,
  title={A fast direct solver for scattering problems involving elongated structures},
  author={Martinsson, Per-Gunnar and Rokhlin, Vladimir},
  journal={J. Comput. Phys.},
  volume={221},
  number={1},
  pages={288--302},
  year={2007},
doi={10.1016/j.jcp.2006.06.037},
  publisher={Elsevier}
}

@article{minden2017recursive,
  title={A recursive skeletonization factorization based on strong admissibility},
  author={Minden, Victor and Ho, Kenneth L and Damle, Anil and Ying, Lexing},
  journal={Multiscale Model. Simul.},
  volume={15},
  number={2},
  pages={768--796},
doi={10.1137/16M1095949},
  year={2017},
  publisher={SIAM}
}

@article{Greengard2009fast,
 title={Fast direct solvers for integral equations in complex three-dimensional domains},
 volume={18},
 DOI={10.1017/S0962492906410011},
 journal={Acta Numer.},
 author={Greengard, Leslie and Gueyffier, Denis and Martinsson, Per-Gunnar and Rokhlin, Vladimir},
 year={2009},
 pages={243--275}
}

@article{kleinman1988single,
  title={On single integral equations for the transmission problem of acoustics},
  author={Kleinman, RE and Martin, PA},
  journal={SIAM J. Appl. Math.},
  volume={48},
  number={2},
  pages={307--325},
  year={1988},
  publisher={SIAM},
doi={10.1137/0148016},
}

@book{eringen1975elastodynamics,
  title={Elastodynamics volume 2, {L}inear {T}heory},
  author={Eringen, A.C. and Suhubi, E.S.},
  year={1975},
address={USA},
doi={10.1016/C2013-0-10628-9},
  publisher={Academic press}
}

@article{CHEN1998529,
title = {On fictitious frequencies using dual series representation},
journal = {Mech. Res. Commun.},
volume = {25},
number = {5},
pages = {529-534},
year = {1998},
issn = {0093-6413},
doi = {10.1016/S0093-6413(98)00069-X},
author = {J.T. Chen}
}

@article{misawa2012jascome,
  title={Boundary integral formulations for one-periodic transmission problems for {H}elmholtz' equation in 2-{D}},
  author={Misawa, Ryota and Nishimura, Naoshi},
  journal={Trans. JASCOME},
  volume={12},
  number={ },
  pages={109-114},
  year={2012},
note={(in Japanese)},
  doi={10.60443/jascome.23.0_71}
}

@article{cheng2005compression,
  title={On the compression of low rank matrices},
  author={Cheng, Hongwei and Gimbutas, Zydrunas and Martinsson, Per-Gunnar and Rokhlin, Vladimir},
  journal={SIAM J. Sci. Comput.},
  volume={26},
  number={4},
  pages={1389--1404},
  year={2005},
doi={10.1137/030602678},
  publisher={SIAM}
}

@article{dominguez2024nystrom,
  title={Nystr{\"o}m discretizations of boundary integral equations for the solution of 2{D} elastic scattering problems},
  author={Dom{\'\i}nguez, V{\'\i}ctor and Turc, Catalin},
  journal={J. Comput. Appl. Math.},
  volume={440},
  pages={115622},
  year={2024},
doi={10.1016/j.cam.2023.115622},
  publisher={Elsevier}
}

@article{tong2007nystrom,
  title={Nystr{\"o}m method for elastic wave scattering by three-dimensional obstacles},
  author={Tong, Mei Song and Chew, Weng Cho},
  journal={J. Comput. Phys.},
  volume={226},
  number={2},
  pages={1845--1858},
  year={2007},
doi={10.1016/j.jcp.2007.06.013},
  publisher={Elsevier}
}

@article{bonnet1998symmetric,
    author = {Bonnet, Marc and Maier, Giulio and Polizzotto, Castrenze},
    title = {Symmetric {G}alerkin Boundary Element Methods},
    journal = {Appl. Mech. Rev.},
    volume = {51},
    number = {11},
    pages = {669-704},
    year = {1998},
    month = {11},
    issn = {0003-6900},
    doi = {10.1115/1.3098983},
}

@article{rokhlin1990rapid,
  title={Rapid solution of integral equations of scattering theory in two dimensions},
  author={Rokhlin, Vladimir},
  journal={J. Comput. Phys.},
  volume={86},
  number={2},
  pages={414--439},
  year={1990},
doi={10.1016/0021-9991(90)90107-C},
  publisher={Elsevier}
}

@article{LeLouer2026muller,
  author    = {Le Lou{\"e}r, Fr{\'e}dr{\'e}rique},
  title     = {On the {M}{\"u}ller boundary integral equation method for solving contact problems in {2D} linear elasticity},
  journal   = {J. Eng. Math.},
  year      = {2026},
  volume    = {159},
  number    = {1},
  pages     = {},
  doi       = {10.1007/s10665-026-10538-y},
  Xurl       = {https://doi.org/10.1007/s10665-026-10538-y}
}

@article{nick2024numerical,
  title={Numerical analysis for electromagnetic scattering with nonlinear boundary conditions},
  author={Nick, J{\"o}rg},
  journal={Math. Comp.},
  volume={93},
  number={348},
  pages={1529--1568},
doi={10.1090/mcom/3914},
  year={2024}
}

@article{ganesh2024fast,
    author = {Ganesh, M. and Hawkins, Stuart C.},
    title = {A fast algorithm for the two-dimensional {H}elmholtz transmission problem with large multiple scattering configurations},
    journal = {J. Acoust. Soc. Am.},
    volume = {156},
    number = {2},
    pages = {752-762},
    year = {2024},
    month = {08},
    Xissn = {0001-4966},
    doi = {10.1121/10.0028121},
    Xurl = {https://doi.org/10.1121/10.0028121},
    Xeprint = {https://pubs.aip.org/asa/jasa/article-pdf/156/2/752/20095977/752_1_10.0028121.pdf},
}

@article{dominguez2022boundary,
  title={Boundary integral equation methods for the solution of scattering and transmission 2{D} elastodynamic problems},
  author={Dom{\'\i}nguez, V{\'\i}ctor and Turc, Catalin},
  journal={IMA J. Appl. Math.},
  volume={87},
  number={4},
  pages={647--706},
  year={2022},
doi={10.1093/imamat/hxac018},
  publisher={Oxford University Press}
}

@Inbook{Ying2009,
author="Ying, Lexing",
editor="Engquist, Bj{\"o}rn
and L{\"o}tstedt, Per
and Runborg, Olof",
title="Fast Algorithms for Boundary Integral Equations",
bookTitle="Multiscale Modeling and Simulation in Science",
year="2009",
publisher="Springer Berlin Heidelberg",
address="Berlin, Heidelberg",
pages="139--193",
isbn="978-3-540-88857-4",
doi="10.1007/978-3-540-88857-4_3",
Xurl="https://doi.org/10.1007/978-3-540-88857-4_3"
}

@article{2019Abduljabbar,
author = {Abduljabbar, Mustafa and Farhan, Mohammed Al and Al-Harthi, Noha and Chen, Rui and Yokota, Rio and Bagci, Hakan and Keyes, David},
title = {Extreme Scale {FMM}-Accelerated Boundary Integral Equation Solver for Wave Scattering},
journal = {SIAM J. Sci. Comput.},
volume = {41},
number = {3},
pages = {C245-C268},
year = {2019},
doi = {10.1137/18M1173599},
}

@article{liu2024massive,
  title={Massive parallelization of multilevel fast multipole algorithm for 3-{D} electromagnetic scattering problems on {SW}26010 many-core cluster},
  author={Liu, Xin-Duo and He, Wei-Jia and Yang, Ming-Lin and Sheng, Xin-Qing},
  journal={J. Supercomput.},
  volume={80},
  number={7},
  pages={8702--8718},
  year={2024},
  publisher={Springer},
doi={https://doi.org/10.1007/s11227-023-05759-2}
}

@book{muller1969foundations,
  title     = {Foundations of the Mathematical Theory of Electromagnetic Waves},
  author    = {M{\"u}ller, Claus},
  year      = {1969},
  publisher = {Springer},
  address   = {Berlin, Heidelberg},
  series    = {Grundlehren der mathematischen Wissenschaften},
  volume    = {155},
  doi       = {10.1007/978-3-662-11773-6},
  isbn      = {978-3-662-11773-6}
}

@article{boubendir2015integral,
  title={Integral equations requiring small numbers of {K}rylov-subspace iterations for two-dimensional smooth penetrable scattering problems},
  author={Boubendir, Yassine and Bruno, Oscar and Levadoux, David and Turc, Catalin},
  journal={Appl. Numer. Math.},
  volume={95},
  pages={82--98},
  year={2015},
doi={10.1016/j.apnum.2015.01.005},
  publisher={Elsevier}
}

@article{hsiao2011system,
  title={A system of boundary integral equations for the transmission problem in acoustics},
  author={Hsiao, George C and Xu, Liwei},
  journal={Appl. Numer. Math.},
  volume={61},
  number={9},
  pages={1017--1029},
  year={2011},
doi={10.1016/j.apnum.2011.05.003},
  publisher={Elsevier}
}

@article{laliena2009symmetric,
  title={Symmetric boundary integral formulations for {H}elmholtz transmission problems},
  author={Laliena, Antonio R and Rap{\'u}n, M-L and Sayas, F-J},
  journal={Appl. Numer. Math.},
  volume={59},
  number={11},
  pages={2814--2823},
  year={2009},
doi={10.1016/j.apnum.2008.12.030},
  publisher={Elsevier}
}

@article{rouet2016distributed,
  title={A distributed-memory package for dense hierarchically semi-separable matrix computations using randomization},
  author={Rouet, Fran{\c{c}}ois-Henry and Li, Xiaoye S and Ghysels, Pieter and Napov, Artem},
  journal={ACM Trans. Math. Software},
  volume={42},
  number={4},
  pages={1--35},
  year={2016},
doi={https://doi.org/10.1145/2930660},
  publisher={ACM New York, NY, USA}
}

@article{bebendorf2000approximation,
  title={Approximation of boundary element matrices},
  author={Bebendorf, Mario},
  journal={Numer. Math.},
  volume={86},
  number={4},
  pages={565--589},
  year={2000},
doi={10.1007/PL00005410},
  publisher={Springer}
}

@article{boukaram2026linear,
  title={Linear complexity {$\mathcal{H}^2$} direct solver for fine-grained parallel architectures},
  author={Boukaram, Wajih and Keyes, David and Li, Xiaoye and Liu, Yang and Turkiyyah, George},
  journal={IMA J. Numer. Anal.},
  volume={drag030},
  year={2026},
doi={10.1093/imanum/drag030},
  publisher={Oxford University Press},
note={(in press)}
}

@article{corona2015n,
  title={An {$O(N)$} direct solver for integral equations on the plane},
  author={Corona, Eduardo and Martinsson, Per-Gunnar and Zorin, Denis},
  journal={Appl. Comput. Harmon. Anal.},
  volume={38},
  number={2},
  pages={284--317},
  year={2015},
doi={10.1016/j.acha.2014.04.002},
  publisher={Elsevier}
}

@article{k_l_ho2016hierarchical,
  title={Hierarchical Interpolative Factorization for Elliptic Operators: Integral Equations},
  author={L. Ho, Kenneth and Ying, Lexing},
  journal={Commun. Pure Appl. Math.},
  volume={69},
  number={7},
  pages={1314--1353},
  year={2016},
doi={10.1002/cpa.21577},
  publisher={Wiley Online Library}
}

@article{matsumoto2026accelerated,
  title={An accelerated direct solver for scalar wave scattering by multiple transmissive inclusions in two dimensions},
  author={Matsumoto, Yasuhiro},
  journal={Mech. Eng. J.},
volume={13},
number={4},
  Xpages={26-00096},
  year={2026},
doi={10.1299/mej.26-00096},
  publisher={The Japan Society of Mechanical Engineers}
}

@article{kandappan2023hodlr2d,
  title={{HODLR2D}: A new class of hierarchical matrices},
  author={Kandappan, VA and Gujjula, Vaishnavi and Ambikasaran, Sivaram},
  journal={SIAM J. Sci. Comput.},
  volume={45},
  number={5},
  pages={A2382--A2408},
  year={2023},
doi={https://doi.org/10.1137/22M1491253},
  publisher={SIAM}
}

@article{costabel1988boundary,
  title={Boundary integral operators on {L}ipschitz domains: elementary results},
  author={Costabel, Martin},
  journal={SIAM J. Math. Anal.},
  volume={19},
  number={3},
  pages={613--626},
  year={1988},
doi={10.1137/0519043},
  publisher={SIAM}
}
